\documentclass[onefignum,onetabnum]{siamart171218}

\usepackage{lipsum}
\usepackage{amsfonts}
\usepackage{graphicx}
\usepackage{epstopdf}
\usepackage{algorithmic}
\usepackage{cancel}
\ifpdf
  \DeclareGraphicsExtensions{.eps,.pdf,.png,.jpg}
\else
  \DeclareGraphicsExtensions{.eps}
\fi

\newsiamremark{remark}{Remark}
\newsiamremark{hypothesis}{Hypothesis}
\crefname{hypothesis}{Hypothesis}{Hypotheses}
\newsiamthm{claim}{Claim}
\newsiamthm{assume}{Assume}

\headers{Entropy Stable Schemes on Dynamic Grids}{N. K. Yamaleev, D. C. Del Ray Fern\'andez, J. Lou, M. H. Carpenter}

\title{Entropy stable spectral collocation schemes for the 3-D Navier-Stokes equations on dynamic unstructured grids
}

\author{Nail K. Yamaleev\thanks{Department of Mathematics and Statistics, Old Dominion University
  (\email{nyamalee@odu.edu}).}
\and David C. Del Rey Fern\'andez\thanks{National Institute of Aerospace (\email{dcdelrey@gmail.com})}
\and Jialin Lou\footnotemark[1]
\and Mark H. Carpenter\thanks{NASA Langley Reserach Center (\email{mark.h.carpeneter@nasa.gov})}
}

\usepackage{amsopn}
\DeclareMathOperator{\diag}{diag}

\usepackage{bm}
\newcommand{\Tr}{\ensuremath{^{\mr{T}}}}
\newcommand{\mr}[1]{\ensuremath{\mathrm{#1}}}
\newcommand{\mat}[1]{\ensuremath{\mathsf{#1}}}
\newcommand{\fnc}[1]{\ensuremath{\mathit{#1}}}
\newcommand{\bfnc}[1]{\ensuremath{\bm{\mathit{#1}}}}

\newcommand{\xm}[0]{\ensuremath{x_{m}}}
\newcommand{\xone}[0]{x_{1}}
\newcommand{\xtwo}[0]{x_{2}}
\newcommand{\xthree}[0]{x_{3}}
\newcommand{\xil}[0]{\ensuremath{\xi_{l}}}
\newcommand{\xione}[0]{\xi_{1}}
\newcommand{\xitwo}[0]{\xi_{2}}
\newcommand{\xithree}[0]{\xi_{3}}
\newcommand{\alphaone}[0]{\ensuremath{\alpha_{1}}}
\newcommand{\alphatwo}[0]{\ensuremath{\alpha_{2}}}
\newcommand{\alphathree}[0]{\ensuremath{\alpha_{3}}}
\newcommand{\betaone}[0]{\ensuremath{\beta_{1}}}
\newcommand{\betatwo}[0]{\ensuremath{\beta_{2}}}
\newcommand{\betathree}[0]{\ensuremath{\beta_{3}}}
\newcommand{\alphal}[0]{\ensuremath{\alpha_{l}}}
\newcommand{\betal}[0]{\ensuremath{\beta_{l}}}

\newcommand{\NOne}[0]{\ensuremath{N_{1}}}
\newcommand{\NTwo}[0]{\ensuremath{N_{2}}}
\newcommand{\NThree}[0]{\ensuremath{N_{3}}}
\newcommand{\Nl}[0]{\ensuremath{N_{l}}}

\newcommand{\bxil}[0]{\ensuremath{\bm{\xi}_{l}}}
\newcommand{\Ck}[0]{\ensuremath{C_{\kappa}}}

\newcommand{\DxiloneD}[0]{\ensuremath{\mat{D}_{\xil}^{(1D)}}}
\newcommand{\HxiloneD}[0]{\ensuremath{\mat{P}_{\xil}^{(1D)}}}
\newcommand{\QxiloneD}[0]{\ensuremath{\mat{Q}_{\xil}^{(1D)}}}
\newcommand{\SxiloneD}[0]{\ensuremath{\mat{S}_{\xil}^{(1D)}}}
\newcommand{\ExiloneD}[0]{\ensuremath{\mat{E}_{\xil}^{(1D)}}}
\newcommand{\txilalpha}[0]{\ensuremath{\bm{t}_{\alphal}}}
\newcommand{\txilbeta}[0]{\ensuremath{\bm{t}_{\betal}}}
\newcommand{\DxioneoneD}[0]{\ensuremath{\mat{D}_{\xione}^{(1D)}}}
\newcommand{\HxioneoneD}[0]{\ensuremath{\mat{P}_{\xione}^{(1D)}}}
\newcommand{\QxioneoneD}[0]{\ensuremath{\mat{Q}_{\xione}^{(1D)}}}

\newcommand{\HxitwooneD}[0]{\ensuremath{\mat{P}_{\xitwo}^{(1D)}}}

\newcommand{\HxithreeoneD}[0]{\ensuremath{\mat{P}_{\xithree}^{(1D)}}}

\newcommand{\txionealpha}[0]{\ensuremath{\bm{t}_{\alphaone}}}
\newcommand{\txionebeta}[0]{\ensuremath{\bm{t}_{\betaone}}}

\newcommand{\Ifive}[0]{\ensuremath{\mat{I}_{5}}}
\newcommand{\onefive}[0]{\ensuremath{\bm{1}_{5}}}
\newcommand{\M}[0]{\ensuremath{\overline{\mat{P}}}}

\newcommand{\Ixitwo}[0]{\ensuremath{\mat{I}_{\xitwo}}}
\newcommand{\Ixithree}[0]{\ensuremath{\mat{I}_{\xithree}}}
\newcommand{\Dxione}[0]{\ensuremath{\overline{\mat{D}}_{\xione}}}
\newcommand{\Qxione}[0]{\ensuremath{\overline{\mat{Q}}_{\xione}}}

\newcommand{\Exione}[0]{\ensuremath{\overline{\mat{E}}_{\xione}}}
\newcommand{\Exionealpha}[0]{\ensuremath{\overline{\mat{E}}_{\alphaone}}}
\newcommand{\Exionebeta}[0]{\ensuremath{\overline{\mat{E}}_{\betaone}}}
\newcommand{\Dxitwo}[0]{\ensuremath{\overline{\mat{D}}_{\xitwo}}}

\newcommand{\Dxithree}[0]{\ensuremath{\overline{\mat{D}}_{\xithree}}}

\newcommand{\Rxilalpha}[0]{\ensuremath{\overline{\mat{R}}_{\alphal}}}
\newcommand{\Rxilbeta}[0]{\ensuremath{\overline{\mat{R}}_{\betal}}}
\newcommand{\Rxilalphasca}[0]{\ensuremath{\mat{R}_{\alphal}}}
\newcommand{\Rxilbetasca}[0]{\ensuremath{\mat{R}_{\betal}}}
\newcommand{\Rxionealpha}[0]{\ensuremath{\overline{\mat{R}}_{\alphaone}}}
\newcommand{\Rxionebeta}[0]{\ensuremath{\overline{\mat{R}}_{\betaone}}}

\newcommand{\Pxilortho}[0]{\ensuremath{\overline{\mat{P}}_{\perp\xil}}}
\newcommand{\Pxilorthosca}[0]{\ensuremath{\mat{P}_{\perp\xil}}}

\newcommand{\Pxioneortho}[0]{\ensuremath{\overline{\mat{P}}_{\perp\xione}}}

\newcommand{\Qxil}[0]{\ensuremath{\overline{\mat{Q}}_{\xil}}}

\newcommand{\Exil}[0]{\ensuremath{\overline{\mat{E}}_{\xil}}}

\newcommand{\Dxil}[0]{\ensuremath{\overline{\mat{D}}_{\xil}}}
\newcommand{\Dxilsca}[0]{\ensuremath{\mat{D}}_{\xil}}

\newcommand{\Msca}[0]{\ensuremath{\mat{P}}}

\newcommand{\Exilsca}[0]{\ensuremath{\mat{E}_{\xil}}}
\newcommand{\matFxmscsca}[3]{\ensuremath{\mat{F}_{\xm}^{sc,(#1)}\left(#2,#3\right)}}
\newcommand{\matQscsca}[3]{\ensuremath{\mat{U}}^{sc,(#1)}\left(#2,#3\right)}

\newcommand{\Ok}[0]{\ensuremath{\Omega_{\kappa}}}
\newcommand{\Gk}[0]{\ensuremath{\Gamma_{\kappa}}}
\newcommand{\Ohat}[0]{\ensuremath{\hat{\Omega}}}
\newcommand{\Ohatk}[0]{\ensuremath{\hat{\Omega}_{\kappa}}}
\newcommand{\Ghatk}[0]{\ensuremath{\hat{\Gamma}_{\kappa}}}
\newcommand{\Ghat}[0]{\ensuremath{\hat{\Gamma}}}
\newcommand{\nxil}[0]{\ensuremath{n_{\xil}}}

\newcommand{\U}[0]{\ensuremath{\bfnc{U}}}
\newcommand{\Fxm}[0]{\ensuremath{\bfnc{F}}_{\xm}}
\newcommand{\GB}[0]{\ensuremath{\bfnc{G}}^{(B)}}
\newcommand{\Gzero}[0]{\ensuremath{\bfnc{G}}^{(0)}}
\newcommand{\Vone}[0]{\ensuremath{\fnc{V}_{1}}}
\newcommand{\Vtwo}[0]{\ensuremath{\fnc{V}_{2}}}
\newcommand{\Vthree}[0]{\ensuremath{\fnc{V}_{3}}}
\newcommand{\Vm}[0]{\ensuremath{\fnc{V}_{m}}}
\newcommand{\Jk}[0]{\ensuremath{\fnc{J}_{\kappa}}}
\newcommand{\Qk}[0]{\ensuremath{\bfnc{U}_{\kappa}}}

\newcommand{\E}[0]{\ensuremath{\fnc{E}}}
\newcommand{\Almk}{\ensuremath{\Jk\frac{\partial\xil}{\partial\xm}}}
\newcommand{\Blk}{\ensuremath{\Jk\frac{\partial\xil}{\partial t}}}
\newcommand{\Alm}{\ensuremath{\fnc{J}\frac{\partial\xil}{\partial\xm}}}
\newcommand{\Bl}{\ensuremath{\fnc{J}\frac{\partial\xil}{\partial t}}}

\newcommand{\matJk}[0]{\ensuremath{\overline{\mat{J}}_{\kappa}}}
\newcommand{\matAlmk}[0]{\ensuremath{\overline{\left[\fnc{J}\frac{\partial\xil}{\partial \xm}\right]}_{\kappa}}}
\newcommand{\matBlk}[0]{\ensuremath{\overline{\left[\fnc{J}\frac{\partial\xil}{\partial t}\right]}_{\kappa}}}
\newcommand{\matBltwolmone}[0]{\ensuremath{\overline{\left[\fnc{J}\frac{\partial\xil}{\partial t}\right]}_{(2l-1)}}}
\newcommand{\matBltwolmonesca}[0]{\ensuremath{\left[\fnc{J}\frac{\partial\xil}{\partial t}\right]_{(2l-1)}}}
\newcommand{\matBlksca}[0]{\ensuremath{\left[\fnc{J}\frac{\partial\xil}{\partial t}\right]_{\kappa}}}
\newcommand{\qk}[0]{\ensuremath{\bm{u}_{\kappa}}}
\newcommand{\qr}[0]{\ensuremath{\bm{u}_{r}}}
\newcommand{\wk}[0]{\ensuremath{\bm{w}_{\kappa}}}
\newcommand{\wtwolmone}[0]{\ensuremath{\bm{w}_{(2l-1)}}}
\newcommand{\wtwol}[0]{\ensuremath{\bm{w}_{2l}}}
\newcommand{\wrr}[0]{\ensuremath{\bm{w}_{r}}}
\newcommand{\psixmk}[0]{\ensuremath{\bm{\psi}_{\xm}^{\kappa}}}

\newcommand{\phik}[0]{\ensuremath{\bm{\varphi}_{\kappa}}}
\newcommand{\phitwolmone}[0]{\ensuremath{\bm{\varphi}_{(2l-1)}}}
\newcommand{\sk}[0]{\ensuremath{\mathcal{S}_{\kappa}}}
\newcommand{\matFxmsc}[2]{\ensuremath{\mat{F}}_{\xm}^{sc}\left(#1,#2\right)}
\newcommand{\matQsc}[2]{\ensuremath{\mat{U}}^{sc}\left(#1,#2\right)}

\newcommand{\one}[0]{\ensuremath{\overline{\bm{1}}}}
\newcommand{\onesca}[0]{\ensuremath{\bm{1}}}
\newcommand{\qtwolmone}[0]{\ensuremath{\bm{u}_{(2l-1)}}}
\newcommand{\qtwol}[0]{\ensuremath{\bm{u}_{2l}}}

\newcommand{\qkmatrix}[2]{\ensuremath{\qk\left(#1, \,#2\right)}}

\newcommand{\fxmsc}[2]{\ensuremath{\bm{\bfnc{F}}_{\xm}^{sc}\left(#1,#2\right)}}
\newcommand{\Qsc}[2]{\ensuremath{\bm{\bfnc{U}}^{sc}\left(#1,#2\right)}}
\newcommand{\qrmatrix}[2]{\ensuremath{\qr\left(#1, \,#2\right)}}
\newcommand{\wkmatrix}[2]{\ensuremath{\wk\left(#1, \,#2\right)}}
\newcommand{\wrmatrix}[2]{\ensuremath{\wrr\left(#1, \,#2\right)}}
\newcommand{\matJksca}[0]{\ensuremath{\mat{J}_{\kappa}}}
\newcommand{\vecJk}[0]{\ensuremath{\bm{J}_{\kappa}}}
\newcommand{\matAlmksca}[0]{\ensuremath{\left[\fnc{J}\frac{\partial\xil}{\partial \xm}\right]_{\kappa}}}

\newcommand{\Gsc}[2]{\ensuremath{\bfnc{G}^{sc}}\left(#1,#2\right)}
\newcommand{\vkmatrix}[2]{\ensuremath{\vk\left(#1, \,#2\right)}}
\newcommand{\vk}[0]{\ensuremath{\bm{v}_{\kappa}}}
\newcommand{\vr}[0]{\ensuremath{\bm{v}_{r}}}
\newcommand{\gk}[0]{\ensuremath{\bm{g}_{\kappa}}}
\newcommand{\gr}[0]{\ensuremath{\bm{g}_{r}}}
\newcommand{\hk}[0]{\ensuremath{\bm{h}_{\kappa}}}
\newcommand{\hr}[0]{\ensuremath{\bm{h}_{r}}}
\newcommand{\vrmatrix}[2]{\ensuremath{\vr\left(#1, \,#2\right)}}
\newcommand{\gkmatrix}[2]{\ensuremath{\gk\left(#1, \,#2\right)}}
\newcommand{\grmatrix}[2]{\ensuremath{\gr\left(#1, \,#2\right)}}

\newcommand{\matGsc}[2]{\ensuremath{\mat{G}^{sc}}\left(#1,#2\right)}
\newcommand{\Dzeta}[0]{\ensuremath{\overline{\mat{D}}_{\zeta}}}
\newcommand{\Ezetasca}[0]{\ensuremath{\mat{E}_{\zeta}}}
\newcommand{\Ezeta}[0]{\ensuremath{\overline{\mat{E}}_{\zeta}}}

\newcommand{\Cij}[2]{\ensuremath{\mat{C}_{#1,#2}}}
\newcommand{\Chatij}[2]{\ensuremath{\hat{\mat{C}}_{#1,#2}}}
\newcommand{\xia}[0]{\ensuremath{\xi_{a}}}
\newcommand{\Fxmv}[0]{\ensuremath{\bm{\fnc{F}}_{\xm}^{(v)}}}
\newcommand{\FI}[0]{\ensuremath{\bm{f}_{I}}}
\newcommand{\SATI}[0]{\ensuremath{\bm{SAT}_{I}}}
\newcommand{\matChatijk}[3]{\ensuremath{\left[\hat{\mat{C}}_{#1,#2}\right]_{#3}}}
\newcommand{\Thetaak}[2]{\ensuremath{\bm{\theta}_{#1}^{#2}}}

\newcommand{\IPk}[0]{\ensuremath{\bm{I}_{\mathrm{P}}^{\kappa}}}
\newcommand{\Dxia}[0]{\ensuremath{\overline{\mat{D}}_{\xi_{a}}}}

\newcommand{\Exia}[0]{\ensuremath{\overline{\mat{E}}_{\xi_{a}}}}
\newcommand{\wtwoa}[0]{\ensuremath{\bm{w}_{2a}}}
\newcommand{\wtwoamone}[0]{\ensuremath{\bm{w}_{(2a-1)}}}
\newcommand{\Rxiaalpha}[0]{\ensuremath{\overline{\mat{R}}_{\alphaa}}}
\newcommand{\Rxiabeta}[0]{\ensuremath{\overline{\mat{R}}_{\betaa}}}
\newcommand{\alphaa}[0]{\ensuremath{\alpha_{a}}}
\newcommand{\betaa}[0]{\ensuremath{\beta_{a}}}
\newcommand{\Pxiaortho}[0]{\ensuremath{\overline{\mat{P}}_{\perp\xia}}}

\ifpdf
\hypersetup{
  pdftitle={Entropy stable spectral collocation schemes for the 3-D Navier-Stokes equations on dynamic unstructured grids},
  pdfauthor={N. K. Yamaleev, D. C. Del Rey Fern\'andez, J. Lou, M. H. Carpenter}
}
\fi

\begin{document}

\maketitle

% REQUIRED
\begin{abstract}
New entropy stable spectral collocations schemes of arbitrary order of accuracy are developed for the unsteady 3-D Euler and Navier-Stokes equations on dynamic unstructured grids. To take into account the grid motion and deformation, we use an arbitrary Lagrangian-Eulerian (ALE) formulation. As a result, moving and deforming hexahedral grid elements are individually mapped onto a cube in the fixed reference system of coordinates. The proposed scheme is constructed by using the skew-symmetric form of the Navier-Stokes equations, which are discretized by using summation-by-parts spectral collocation operators that preserve the conservation properties of the original governing equations. Furthermore, the metric coefficients are approximated such that the geometric conservation laws (GCL) are satisfied exactly on both static and dynamic grids. To make the scheme entropy stable, a new entropy conservative flux  is derived for the 3-D Euler and Navier-Stokes equations on dynamic unstructured grids. The new flux preserves the design order of accuracy of the original spectral collocation scheme and guarantees the entropy conservation on moving and deforming grids. We present numerical results demonstrating design order of accuracy and freestream preservation properties of the new schemes for both the Euler and Navier-Stokes equations on moving and deforming unstructured grids.
\end{abstract}

% REQUIRED

% REQUIRED
\section{Introduction}
%=============
In many fluid dynamics applications including fluid-structure interaction,  maneuvering, free-surface flows, and others, the domain boundaries undergo large
displacements and deformations. One of the most widely used approaches for solving this class of problems is based on the arbitrary Lagrangian-Eulerian (ALE) 
framework that maps the Navier-Stokes equations in a reference frame wherein the computational grid can move and deform independently of the fluid flow 
\cite{Donea1982, Hughes1981}. It should be noted that the grid motion and deformation may destroy both the conservation and stability properties of numerical methods based on the ALE formulation \cite{Farhat1996}, thus indicating that a rigorous analysis is required for constructing stable ALE-based numerical schemes for the Euler and Navier-Stokes equations in moving and deforming domains.

In the ALE methodology, it is desirable to represent the Euler and Navier-Stokes equations
in strong conservation form to satisfy the conservation laws \ in time--dependent curvilinear coordinates. 
The continuous Euler and Navier-Stokes equations can be recast in strong conservation form if the
geometric conservation law (GCL) equations are satisfied, which is always the case if a mapping between
the physical and time-dependent curvilinear coordinates is diffeomorphism (e.g., see \cite{Fisher2012phd}). 
Though the GCL equations are satisfied exactly at the continuous level, this is not necessarily the case at the discrete level. 
A concept of the discrete geometric conservation laws was first introduced by Thomas and Lombard in \cite{Thomas1979}.
To satisfy the discrete GCL equations exactly and improve the accuracy of the numerical solution on moving and deforming grids, 
the metric coefficients in \cite{Thomas1979} are first recast in conservation law form and 
then discretized by the same finite difference scheme used for approximating the Navier-Stokes equations. 

It should be noted that satisfying the discrete GCL equations improves not only the accuracy 
but also stability of a numerical scheme on moving grids.  Formaggia and Nobile show that for a linear
advection-diffusion equation,  the first-order discrete geometric  conservation law is a sufficient condition
for the backward Euler implicit scheme to be unconditionally stable. In \cite{Farhat2001}, it is proven
that for a nonlinear scalar hyperbolic conservation law equation, the discrete GCL equations are necessary 
and sufficient condition to guarantee that the 1st- and 2nd-order total variation diminishing (TVD) finite volume schemes
are nonlinearly stable in the sense that they satisfy the discrete maximum principle.
 
Recently, provably stable ALE formulations have been developed for high-order finite difference, discontinuous Galerkin, and spectral collocation methods. Nikkar and Nordstr\"om use SBP operators both in time and space to construct high-order, fully discrete, conservative, energy stable
finite-difference schemes for the skew-symmetric form of linear constant-coefficient hyperbolic systems in deforming domains \cite{Nordstrom2015}. This approach is extended to semi-discrete discontinuous Galerkin spectral collocation discretizations in \cite{Kopriva2016}.
For nonlinear conservation laws, the current state of the art is entropy stable schemes on static meshes.
In \cite{Fisher2012phd, Carpenter2014, CFNPSY2016}, high-order  entropy stable finite-difference and spectral collocation schemes based on summation-by-parts (SBP) operators  that satisfy the discrete GCL equations  are constructed for the 3-D unsteady Navier-Stokes equations on static curvilinear grids. These entropy stable methods have recently been extended to entropy stable spectral collocation 
weighted essentially non-oscillatory (WENO) \cite{Yamaleev2017} and  fully discrete schemes \cite{Friedrich2018} on static grids.

Despite the progress in  ALE formulations, so far, no high-order provably stable numerical schemes have been developed
for the unsteady compressible Euler and Navier-Stokes equations on dynamic grids. To address this problem, we develop
new high-order spectral collocation schemes that are entropy stable, conservative, and satisfy the discrete GCL equations exactly on moving 
and deforming unstructured grids. The proposed scheme is constructed by discretizing  the skew-symmetric form of the Navier-Stokes equations using SBP operators. A new entropy conservative flux is derived for the 3-D Euler and Navier-Stokes equations. 
The new flux preserves the design order of accuracy of the baseline scheme and provides entropy stability on moving and deforming grids. 

The paper is organized as follows. The notation used in the paper is laid out in \cref{sec:notation}. 
The continuous geometric conservation law (GCL) equations and entropy stability analysis of the Navier-Stokes equations in moving and deforming domains are reviewed in \cref{sec:continuous}. The numerical algorithm and its analysis are presented in \cref{sec:discrete}. Implementational details are discussed in 
\cref{sec:implementation}, while results are presented in \cref{sec:results}. Finally, conclusions are drawn in \cref{sec:conclusions}.

\section{Notation and definitions}\label{sec:notation}
%================================
Partial differential equations (PDEs) under consideration are discretized on cubes having Cartesian computational coordinates denoted by 
the 3-tuple $(\xione,\xitwo,\xithree)$, where the physical coordinates are denoted by the 3-tuple 
$(\xone,\xtwo,\xthree)$. Vectors are represented by lowercase bold font, for example 
\begin{equation*}
\bm{\xi}_{1} = \left[\xi_{1,1},\dots,\xi_{1,N_{1}}\right]\Tr,
\end{equation*}
while matrices are represented using sans-serif font, for example, $\mat{A}$. Continuous 
scalar functions on a space-time domain are denoted by capital letters in script font.  For example, 
\begin{equation*}
\fnc{U}\left(\xione,\xitwo,\xithree,t\right)\in L^{2}\left(\left[\alphaone,\betaone\right]\times
\left[\alphatwo,\betatwo\right]\times\left[\alphathree,\betathree\right]\times\left[0,T\right]\right)
\end{equation*}
represents a square integrable function, where $t$ is the temporal coordinate. The restriction of such 
functions onto a set of mesh nodes is denoted by lower case bold font. For example, the restriction of $\fnc{U}$ onto a grid of 
$\NOne\times\NTwo\times\NThree$ nodes is given by the vector
\begin{equation*}
\bm{u} = \left[\fnc{U}\left(\xi_{1,1},\xi_{2,1},\xi_{3,1},t\right),\dots,\fnc{U}\left(\xi_{1,\NOne},
\xi_{2,\NTwo},\xi_{3,\NThree},t\right)\right]\Tr.
\end{equation*}
One-dimensional SBP operators are used to discretize derivatives and are extended to multiple dimensions using tensor 
products. 
\begin{definition}\label{def:tens}
The tensor product between the matrix $\mat{A}\in\mathbb{R}^{n,m}$ and $\mat{B}\in\mathbb{R}^{p,q}$ 
is denoted by 

\begin{equation*}
\mat{A}\otimes\mat{B}\equiv
\left[
\begin{array}{ccc}
a_{1,1}\mat{B}&\dots&a_{1,m}\mat{B}\\
\vdots&&\vdots\\
a_{n,1}\mat{B}&\dots&a_{n,m}\mat{B}
\end{array}
\right].
\end{equation*}
\end{definition}
For basic properties of tensor products, the reader is referred to~\cite{Horn2012}.
%
%and has the following useful properties:
%\begin{itemize}
%\item $\left(\mat{A}\otimes\mat{B}\right)\Tr=\mat{A}\Tr\otimes\mat{B}\Tr$
%\item  $\left(\mat{A}\otimes\mat{B}\right)^{-1}=\mat{A}^{-1}\otimes\mat{B}^{-1}$
%\item $\left(\mat{A}\otimes\mat{B}\right)\left(\mat{C}\otimes\mat{D}\right)=\mat{A}\mat{C}\otimes\mat{B} \mat{D}$
%provided that $\mat{A}\mat{C}$ and $\mat{B}\mat{D}$ exist 
%\end{itemize}
%\end{definition}
%
The Hadamard product is used in the construction of entropy conservative/stable discretizations and is defined by
\begin{definition}\label{def:had}
The Hadamard product between the matrix $\mat{A}\in\mathbb{R}^{n\times n}$ and $\mat{B}\in\mathbb{R}^{n\times n}$ is denoted by 
\begin{equation*}
\mat{A}\circ\mat{B} \equiv 
\left[
\begin{array}{ccc}
a_{11}b_{11}&\dots&a_{1n}b_{1n}\\
\vdots&&\vdots\\
a_{n1}b_{n1}&\dots&a_{nn}b_{nn}
\end{array}
\right].
\end{equation*}
\end{definition}
%
%and has the useful properties
%
%\begin{equation*}
%\begin{split}
%&\mat{A}\circ\mat{B}=\mat{B}\circ\mat{A},\\
%&\mat{A}\circ\left(\mat{B}\circ\mat{C}\right)=\left(\mat{A}\circ\mat{B}\right)\circ\mat{C},\\
%&\mat{A}\circ\left(\mat{B}+\mat{C}\right)=\mat{A}\circ\mat{B}+\mat{A}\circ\mat{C},\\
%&\left(\mat{A}\circ\mat{B}\right)\Tr=\mat{A}\Tr\circ\mat{B}\Tr.
%\end{split}
%\end{equation*}
%
%\end{definition}
For basic properties of Hadamard products, we refer the reader to~\cite{Horn2012}.

SBP operators are matrix difference operators that approximate derivatives at mesh nodes and are 
characterized by their exactness in differentiating monomials, thus, an order $p$ SBP operator is 
one that differentiates exactly monomials up to degree $p$. 
Oftentimes monomials are discussed and the following notation is used:
\[
\bxil^{k} \equiv \left[\xi_{l,1}^{k},\dots,\xi_{l,\Nl}^{k}\right]\Tr,
\]
with the convention that $\bxil^{k}=\bm{0}$ for $k<0$.

We now give the definition of a one-dimensional SBP operator in the $\xil$ direction, $l=1,2,3$ (also, see \cite{DCDRF2014,Fernandez2014,Svard2014}).
\begin{definition}\label{SBP}
\textbf{Summation-by-parts (SBP) operator for the first \newline derivative}: A matrix operator, 
$\DxiloneD\in\mathbb{R}^{\Nl\times\Nl}$, is an SBP operator of order $p$ approximating the derivative 
$\frac{\partial}{\partial \xil}$ on the domain $\xil\in\left[\alphal,\betal\right]$ with nodal 
distribution $\bm{\xi}_{l}$ having $\Nl$ nodes, if 
%\begin{enumerate}
%\item $\DxiloneD\bxil^{k}=\left(\HxiloneD\right)^{-1}\QxiloneD\bxil^{k}=
%\left(\HxiloneD\right)^{-1}\left(\SxiloneD+\frac{1}{2}\ExiloneD\right)\bxil=k\bxil^{k-1}$, 
%$k=0,1,\dots,p$;
%\item $\HxiloneD$, denoted the norm matrix, is symmetric positive definite;
%\item $\ExiloneD=\left(\ExiloneD\right)\Tr$, $\SxiloneD=-\left(\SxiloneD\right)\Tr$, therefore, 
%$\QxiloneD+\left(\QxiloneD\right)\Tr=\ExiloneD$; and 
%\item $\left(\bxil^{i}\right)\Tr\ExiloneD\bxil^{j} = \betal^{i+j}-\alphal^{i+j}$, $i,\,j=0,1,\dots,r$, 
%$r\ge p$.
%\end{enumerate}
\begin{enumerate}
\item $\DxiloneD\bxil^{k}=\left(\HxiloneD\right)^{-1}\QxiloneD\bxil^{k}=k\bxil^{k-1}$, $k=0,1,\dots,p$;
\item $\DxiloneD=\left(\HxiloneD\right)^{-1}\QxiloneD$, where the norm matrix, $\HxiloneD$, is symmetric and positive definite;
\item $\QxiloneD=\left(\SxiloneD+\frac{1}{2}\ExiloneD\right)$, where $\SxiloneD=-\left(\SxiloneD\right)\Tr$, $\ExiloneD=\left(\ExiloneD\right)\Tr$, 
and $\ExiloneD$ satisfies 
\[\left(\bxil^{i}\right)\Tr\ExiloneD\bxil^{j} = \betal^{i+j}-\alphal^{i+j},\qquad i,\,j=0,1,\dots,r,\qquad r\ge p.\]
\end{enumerate}
\end{definition}
In order to impose boundary conditions or inter-element coupling using 
simultaneous-approximation-term (SAT) penalty type conditions \cite{Carpenter1999}, it is useful to 
decompose $\ExiloneD$ as the outer product of interpolation/extrapolation operators
\cite{DCDRF2014}:
\begin{equation*}
\ExiloneD = \txilbeta\txilbeta\Tr-\txilalpha\txilalpha\Tr.
\end{equation*}
The interpolation/extrapolation operators $\txilalpha$ and $\txilbeta$ are $(r+1)$th--order  accurate and therefore, satisfy the following 
accuracy conditions:
\begin{equation*}
\txilalpha\Tr\bxil^{k}=\alphal^{k},\quad\txilbeta\Tr\bxil^{k}=\betal^{k},\quad k=0,1,\dots,r.
\end{equation*}
Here, the three-dimensional Euler and Navier-Stokes equations are considered and the required derivative 
operators are given as
\begin{equation}\label{recastderiv}
\begin{array}{ll}
\M\equiv\HxioneoneD\otimes\HxitwooneD\otimes\HxithreeoneD\otimes\Ifive,&\\
\\
\Dxione\equiv\M^{-1}\Qxione,&\Qxione\equiv\QxioneoneD\otimes\HxitwooneD\otimes\HxithreeoneD\otimes\Ifive,\\
%\\
%\Dxitwo\equiv\M^{-1}\Qxitwo,&\Qxitwo\equiv\HxioneoneD\otimes\QxitwooneD\otimes\HxithreeoneD\otimes\Ifive,\\
%\\
%\Dxithree\equiv\M^{-1}\Qxithree,&\Qxithree\equiv\HxioneoneD\otimes\HxitwooneD\otimes\QxithreeoneD\otimes\Ifive,
\end{array}
\end{equation}
where $\Ifive$ is a $5\times 5$ identity matrix and $\Dxitwo$ and $\Dxithree$ are defined similarly. The surface matrices are given as
\begin{equation}\label{recastE}
\begin{array}{lll}
\Exione\equiv\Exionebeta+\Exionealpha,
&\Exionealpha\equiv-\Rxionealpha\Tr\Pxioneortho\Rxionealpha,
&\Exionebeta\equiv\Rxionebeta\Tr\Pxioneortho\Rxionebeta,
\\\\
\Rxionealpha\equiv\txionealpha\Tr\otimes\Ixitwo\otimes\Ixithree\otimes\Ifive,
&\Rxionebeta\equiv\txionebeta\Tr\otimes\Ixitwo\otimes\Ixithree\otimes\Ifive,
&\Pxioneortho\equiv\HxitwooneD\otimes\HxithreeoneD\otimes\Ifive,
%\Exitwo\equiv\Exitwobeta+\Exitwoalpha,
%&\Exitwoalpha\equiv-\Rxitwoalpha\Tr\Pxitwoortho\Rxitwoalpha
%&\Exitwobeta\equiv\Rxitwobeta\Tr\Pxitwoortho\Rxitwobeta,\\
%\\
%\Rxitwoalpha\equiv\Ixione\otimes\txitwoalpha\Tr\otimes\Ixithree\otimes\Ifive,
%&\Rxitwobeta\equiv\Ixione\otimes\txitwobeta\Tr\otimes\Ixithree\otimes\Ifive,
%&\Pxitwoortho\equiv\HxioneoneD\otimes\HxithreeoneD\otimes\Ifive,\\\\
%\Exithree\equiv\Exithreebeta+\Exithreealpha,
%&\Exithreealpha\equiv-\Rxithreealpha\Tr\Pxithreeortho\Rxithreealpha,
%&\Exithreebeta\equiv\Rxithreebeta\Tr\Pxithreeortho\Rxithreebeta,\\
%\\
%\Rxithreealpha\equiv\Ixione\otimes\Ixitwo\Tr\otimes\txithreealpha\Tr\otimes\Ifive,
%&\Rxithreebeta\equiv\Ixione\otimes\Ixitwo\Tr\otimes\txithreebeta\Tr\otimes\Ifive,
%&\Pxithreeortho\equiv\HxioneoneD\otimes\HxitwooneD\otimes\Ifive.
\end{array}
\end{equation}
and the operators in the other two computational directions are defined similarly.
\begin{remark}
It is at times necessary to consider scalar SBP operators and in such a case, the overbar notation is dropped, for example, 
\[
\mat{D}_{\xione}\equiv\DxioneoneD\otimes\Ixitwo\otimes\Ixithree.
\]
\end{remark}
\begin{remark}
While tensor-product SBP operators are used, all the theoretical results are presented in terms of multidimensional SBP operators 
and are therefore general. Along these lines, throughout the text, general nodal distributions are referenced as 
$\Ck\equiv\left\{\bm{\xi}_{i}\right\}_{i}^{N}$.
\end{remark}
The physical domain $\Omega\subset\mathbb{R}^{3}$, 
with boundary $\Gamma\equiv\partial\Omega$, with Cartesian coordinates $\left(\xone,\xtwo,\xthree\right)\subset\mathbb{R}^{3}$, 
is partitioned into $K$ nonoverlapping hexahedral elements. The domain of the $\kappa^{\text{th}}$ element is denoted by 
$\Ok$ and has boundary $\Gk$. Numerically, PDEs are solved in computational 
coordinates, $\left(\xione,\xitwo,\xithree\right)\subset\mathbb{R}^{3}$, 
where each $\Ok$ is locally transformed to $\Ohatk$, with boundary $\Ghat\equiv\partial\Ohatk$, under the following 
assumption:
\begin{assume}\label{assume:curv}
Each element in physical space is transformed using 
a local, smooth, one-to-one, and invertible curvilinear coordinate transformation that is compatible at 
shared interfaces;
% meaning that points in computational space on either side of a 
% shared interface  mapped to the same physical location and therefore map back 
% to the analogous location in computational space; 
thus, a grid generated by these local curvilinear coordinate transformations is watertight.
\end{assume}
Associating the index 3-tuple $(i,j,k)$ with the computational coordinate 3-tuple \newline $(\xione, \xitwo, \xithree)$, 
the face numbering convention is given in Table~\ref{tab:naming}.
\begin{table}[!t]
\begin{center}
\begin{tabular}{lc} 
index& face number\\ 
\hline
$i = 1$ & 1 \\
$i =N$ & 2 \\
$j = 1$ & 3 \\
$j = N$ & 4 \\
$k = 1 $&5 \\
$k = N$ & 6\\
\hline 
\end{tabular}
\caption{Face numbering convention}
\label{tab:naming}
\end{center}
\end{table}

\section{Continuous analysis}\label{sec:continuous}
%===============================
\subsection{Governing equations}
%----------------------------------------
The 3-D compressible Navier-Stokes equations in conservation law form  in the Cartesian coordinates $(\xone,\xtwo,\xthree)$ are given by
\begin{equation}\label{eq:NS_C}
\begin{split}
&\frac{\partial\U}{\partial t}+\sum\limits_{m=1}^{3}\frac{\partial \Fxm}{\partial \xm} = \sum\limits_{m=1}^{3}\frac{\partial \Fxmv}{\partial \xm}, \quad 
\forall \left(\xone,\xtwo,\xthree\right)\in\Omega,\quad t\ge 0,\\
&\U\left(\xone,\xtwo,\xthree,t\right)=\GB\left(\xone,\xtwo,\xthree,t\right),\quad\forall \left(\xone,\xtwo,\xthree\right)\in\Gamma,\quad t\ge 0,\\
&\U\left(\xone,\xtwo,\xthree,0\right)=\Gzero\left(\xone,\xtwo,\xthree,0\right), \quad 
\forall \left(\xone,\xtwo,\xthree\right)\in\Omega,
\end{split}
\end{equation}
where $\U$ is a vector of conservative variables, and  $\Fxm$, and $\Fxmv$ are the inviscid and viscous fluxes associated with the $\xm$ coordinate, respectively. 
The boundary data, $\GB$ and the initial condition, $\Gzero$, are assumed to be bounded in $L^{2}\cap L^{\infty}$. In 
addition, $\GB$ is assumed to contain boundary data that are stable in the entropy sense. 

The vector of conservative variables is given as 
\begin{equation*}
\U = \left[\rho,\rho\Vone,\rho\Vtwo,\rho\Vthree,\rho\E\right]\Tr,
\end{equation*}
where $\rho$ denotes the density, $\bfnc{V} = \left[\Vone,\Vtwo,\Vthree\right]\Tr$ is the velocity 
vector, and $\E$ is the specific total energy. The inviscid fluxes, $\Fxm, m=1,2,3$, are given by
\begin{equation*}
\Fxm = \left[\rho\Vm,\rho\Vm\Vone+\delta_{m,1}\fnc{P},\rho\Vm\Vtwo+\delta_{m,2}\fnc{P},\rho\Vm\Vthree+\delta_{m,3}\fnc{P},\rho\Vm\fnc{H}\right]\Tr, 
\end{equation*}
where $\fnc{P}$ is the pressure, $\fnc{H}$ is the specific total enthalpy and $\delta_{i,j}$ is the 
Kronecker delta (we have broken our convention for continuous scalar functions in the definition of $\rho$).

The viscous fluxes, $\Fxmv, m=1,2,3$, are defined as
\begin{equation}\label{eq:Fv}
\Fxmv=\left[0,\tau_{1,m},\tau_{2,m},\tau_{3,m},\sum\limits_{i=1}^{3}\tau_{i,m}\fnc{U}_{i}-\kappa\frac{\partial \fnc{T}}{\partial\xm}\right]\Tr.
\end{equation} 
The viscous stresses are given by
\begin{equation}\label{eq:tau}
\tau_{i,j} = \mu\left(\frac{\partial\fnc{V}_{i}}{\partial x_{j}}+\frac{\partial\fnc{V}_{j}}{\partial x_{i}}
-\delta_{i,j}\frac{2}{3}\sum\limits_{n=1}^{3}\frac{\partial\fnc{V}_{n}}{\partial x_{n}}\right),
\end{equation}
where $\mu(T)$ is the dynamic viscosity and $\kappa(T)$ is the thermal conductivity (not to be confused with the choice of 
parameter for element numbering). 

To close the Navier-Stokes equations, Eq. (\ref{eq:NS_C}), the following constituent relations are used:
\begin{equation*}
\fnc{H} = c_{\fnc{P}}\fnc{T}+\frac{1}{2}\bfnc{v}\Tr\bfnc{v},\quad \fnc{P} = \rho R \fnc{T},\quad R = \frac{R_{u}}{M_{w}},
\end{equation*}
where $\fnc{T}$ is the temperature, $R_{u}$ is the universal gas constant, $M_{w}$ is the molecular weight of the gas, 
and $c_{\fnc{P}}$ is the specific heat capacity at constant pressure. Finally, the specific thermodynamic entropy is given as 
\begin{equation}
\label{eq:tentropy}
s=\frac{R}{\gamma-1}\log\left(\frac{\fnc{T}}{\fnc{T}_{\infty}}\right)-R\log\left(\frac{\rho}{\rho_{\infty}}\right),\quad \gamma=\frac{c_{p}}{c_{p}-R},
\end{equation}
where $\fnc{T}_{\infty}$ and $\rho_{\infty}$ are reference temperature and density, respectively
(the stipulated convention has been broken here and $s$ has been used rather than ${S}$ for reasons that will be discussed in \cref{sec:continuous_entropy}). 

\subsection{Geometric Conservation Law (GCL) Equations}
%-----------------------------------------------------------------------
\label{sec:GCL}
To solve the Navier-Stokes equations in complex geometries with moving and deforming boundaries, we use 
the arbitrary Lagrangian-Eulerian (ALE) formulation. A dynamic unstructured grid in the physical domain is generated by individually mapping a reference domain $(\xi_1, \xi_2, \xi_3) \in \hat{\Omega} = [-1,1]^3$ with time $\tau$ onto each grid element in the physical domain $(x_1, x_2, x_3) \in \Omega$ with time $t$.
Assuming that each individual transformation 
\begin{equation*}
\begin{array}{l l}
t         & = \tau, \\
\bm{x} & = \bm{x}(\tau, \bm{\xi}),
\end{array}
\end{equation*}
is a diffeomorphism, it can be described by the following Jacobian matrix: 
\begin{equation*}
\frac{\partial(t,\bm{x})}{\partial(\tau,\bm{\xi})} =\left[
\begin{array}{c c c c}
1  &  0  &  0  &  0  \\
\frac{\partial x_1}{\partial \tau}  &  \frac{\partial x_1}{\partial \xi_1} & \frac{\partial x_1}{\partial \xi_2} & \frac{\partial x_1}{\partial \xi_3}  \\
\frac{\partial x_2}{\partial \tau}  &  \frac{\partial x_2}{\partial \xi_1} & \frac{\partial x_2}{\partial \xi_2} & \frac{\partial x_2}{\partial \xi_3}  \\
\frac{\partial x_3}{\partial \tau}  &  \frac{\partial x_3}{\partial \xi_1} & \frac{\partial x_3}{\partial \xi_2} & \frac{\partial x_3}{\partial \xi_3} 
\end{array}
\right], \quad J = \left| \frac{\partial(t,\bm{x})}{\partial(\tau,\bm{\xi})} \right| .
\end{equation*}

Taking into account that 
\begin{equation*}
\begin{split}
&\frac{\partial}{\partial t}      = \frac{\partial}{\partial \tau} + \sum\limits_{l=1}^3 \frac{\partial \xi_l}{\partial t}\frac{\partial}{\partial \xi_l}, \\
&\frac{\partial}{\partial x_m} = \sum\limits_{l=1}^3 \frac{\partial \xi_l}{\partial x_m}\frac{\partial}{\partial \xi_l}, \quad m=1, 2, 3,
\end{split}
\end{equation*}
and multiplying Eq. (\ref{eq:NS_C}) by the metric Jacobian, $J$, the Navier-Stokes equations can be recast 
in the time-dependent curvilinear coordinates $(\xi_1, \xi_2, \xi_3, \tau)$ as follows:
\begin{equation}
\label{eq:NS_crv_NC}
J\frac{\partial\U}{\partial \tau} + \sum\limits_{l=1}^{3}J\frac{\partial \xi_l}{\partial t}\frac{\partial \U}{\partial \xi_l} 
+  \sum\limits_{m, l=1}^{3}J\frac{\partial \xi_l}{\partial \xm}\frac{\partial \Fxm}{\partial \xi_l} = 
\sum\limits_{m,l=1}^{3}J\frac{\partial \xi_l}{\partial \xm}\frac{\partial \Fxmv}{\partial \xi_l}.
\end{equation}
In contrast to the original Navier-Stokes equations in the Cartesian coordinates,  Eq. (\ref{eq:NS_crv_NC}) is not written in divergence from.
Using the product rule, the Navier-Stokes equations in the curvilinear coordinates can be represented in conservation law from as follows:
\begin{equation}
\label{eq:NS_crv_C}
\begin{array}{l}
\frac{\partial J\U}{\partial \tau} + \sum\limits_{l=1}^{3}\frac{\partial}{\partial \xi_l}\left( J\frac{\partial \xi_l}{\partial t}\U\right)
+  \sum\limits_{m, l=1}^{3}\frac{\partial}{\partial \xi_l}\left(J\frac{\partial \xi_l}{\partial \xm}\left(\Fxm-\Fxm^{(v)}\right)\right)                   \\
-\U\left[\frac{\partial J}{\partial \tau} + \sum\limits_{l=1}^3 \frac{\partial}{\partial \xi_l}\left( J \frac{\partial \xi_l}{\partial t} \right) \right]
-  \left(\Fxm-\Fxm^{(v)}\right) \sum\limits_{m, l=1}^{3}\frac{\partial}{\partial \xi_l}\left(J\frac{\partial \xi_l}{\partial \xm}\right) = 0.
\end{array}
\end{equation}
Based on the definition of the metric coefficients, the following identities hold:
\begin{equation}
\label{eq:GCL_C}
\begin{array}{l}
\frac{\partial J}{\partial \tau} + \sum\limits_{l=1}^3 \frac{\partial}{\partial \xi_l}\left( J \frac{\partial \xi_l}{\partial t} \right) = 0, \\
\sum\limits_{ l=1}^{3}\frac{\partial}{\partial \xi_l}\left(J\frac{\partial \xi_l}{\partial \xm}\right) = 0, \quad m=1, 2, 3,
\end{array}
\end{equation}
which are called the geometric conservation law (GCL) equations. Taking into account the above identities, the last two terms in Eq. (\ref{eq:NS_crv_C})
vanish and the Navier-Stokes equations can be rewritten in the following fully conservative form:
\begin{equation}
\label{eq:NS_CC}
\frac{\partial J\U}{\partial \tau} + \sum\limits_{l=1}^{3}\frac{\partial}{\partial \xi_l}\left( J\frac{\partial \xi_l}{\partial t}\U\right)
+  \sum\limits_{m, l=1}^{3}\frac{\partial}{\partial \xi_l}\left(J\frac{\partial \xi_l}{\partial \xm}\left(\Fxm-\Fxm^{(v)}\right)\right) = 0.
\end{equation}
Note that the GCL equations (\ref{eq:GCL_C}) guarantee that any physically meaningful constant vector of conservative variables 
$\U = \bm{const}$ is a solution of the Navier-Stokes equations \eqref{eq:NS_CC}.  
Though, the GCL equations \eqref{eq:GCL_C} are satisfied exactly at the continuous level,
this is not the case at the discrete level \cite{Thomas1979}. A detailed analysis on how the Navier-Stokes equations and the corresponding metric 
coefficients should be discretized to satisfy the GCL equations is presented in \cref{sec:discrete}.

\subsection{Continuous Entropy Analysis}\label{sec:continuous_entropy}
%----------------------------------------------------------------------------------------
A necessary condition for selecting a unique, physically
relevant solution among possibly many weak solutions of Eq. (\ref{eq:NS_C})  is the entropy
inequality. It is well known that the entropy inequality holds for the Navier-Stokes equations in the 
Cartesian and curvilinear coordinates that are independent of time (e.g., see \cite{CFNPSY2016}). In the present
analysis, we show that the entropy inequality also holds for the Navier-Stokes equations
in time-dependent curvilinear coordinates. 
 
The compressible Navier-Stokes equations are equipped with a convex scalar entropy function $\mathcal{S}$ and the corresponding entropy flux $\mathcal{F}$,
which are given by
\begin{equation}
\begin{array}{l}
\mathcal{S} = -\rho s, \\
\mathcal{F} = -\rho s \bm{V},
\end{array}
\end{equation}
where $s$ is the thermodynamic entropy defined by Eq. (\ref{eq:tentropy}) and $\bm{V}$ is the velocity vector.
Note that the mathematical entropy $\mathcal{S}$ has the opposite sign from the thermodynamic
entropy. Thus, the mathematical entropy across a shock decreases rather than increases. This nomenclature is used throughout 
the paper.

The entropy function  $\mathcal{S}$ satisfies the following properties:
\begin{enumerate}
\item 
$\mathcal{S}(\U)$ is convex and its Hessian matrix, $\frac{\partial^2 \mathcal{S}}{\partial U^2}$, is positive definite provided that $\rho>0$ and $T>0$ $\forall \bm{x} \in \Omega$, thus yielding 
a one-to-one mapping from the conservative to entropy variables that are defined as follows: 
\begin{equation}
\label{eq:Evariables}
\bm{W}\Tr \equiv \frac{\partial \mathcal{S}}{\partial \U}= \left[ \frac hT -s -\frac{\bm{V}\Tr\bm{V}}{2T}, \frac{\Vone}T, \frac{\Vtwo}T, \frac{\Vthree}T, -\frac 1T \right]\Tr.
\end{equation}
\item
The entropy variables satisfy the following compatibility relations for all inviscid fluxes of the compressible Navier-Stokes equations:
\begin{equation}
\label{eq:contract}
\bm{W}\Tr \frac{\partial \Fxm}{\partial x_m} = \bm{W}\Tr \frac{\partial \Fxm}{\partial \U} \frac{\partial \U}{\partial x_m}=\frac{\partial \mathcal{F}_{x_m}}{\partial U}\frac{\partial \U}{\partial x_m} = \frac{\partial \mathcal{F}_{x_m}}{\partial x_m}, \quad m=1, 2, 3,
\end{equation}
where $\mathcal{F}_{x_m}$ is the entropy flux in the $m$-th spatial direction.
\item
The entropy variables symmetrize the compressible Navier-Stokes equations, which can be recast in terms of $\bm{W}$ as follows:
\begin{equation}\label{eq:NS_symmetric}
\begin{split}
&\frac{\partial \U}{\partial \bfnc{W}}\frac{\partial \bfnc{W}}{\partial t}+ \sum\limits_{m=1}^{3} \frac{\partial \Fxm}{\partial \bfnc{W}}\frac{\partial \bfnc{W}}{\partial \xm} =\sum\limits_{l,m=1}^{3}\frac{\partial}{\partial x_l}\left(\Cij{l}{m}\frac{\partial\bfnc{W}}{\partial \xm}\right),
\end{split}
\end{equation}
with the symmetry conditions $\frac{\partial U}{\partial W} = \left(\frac{\partial U}{\partial W}\right)\Tr$, $\frac{\partial \Fxm}{\partial U} = \left( \frac{\partial \Fxm}{\partial U} \right)\Tr$, and $\Cij{l}{m} = (\Cij{l}{m})\Tr$. 
Furthermore, $\frac{\partial U}{\partial W}$ is positive definite, and the matrices $\Cij{l}{m}$ are positive semi-definite, provided that $\rho>0$ and $T>0$ $\forall \bm{x} \in \Omega$.
Note that the term on the right--hand side of Eq. (\ref{eq:NS_symmetric}) is a recast of the viscous fluxes in terms of entropy variables, that is, 
\begin{equation}\label{eq:Fxment}
\Fxmv=\sum\limits_{j=1}^{3}\Cij{m}{j}\frac{\partial\bfnc{W}}{\partial x_{j}}.
\end{equation}
\end{enumerate}
It has been proven by Godunov in \cite{Godunov1961} that if Eq. (\ref{eq:NS_C}) is symmetrized by introducing new variables $\bm{W}$ and
 $\bm{W}$ is a convex function of $\varphi$, then the entropy function and the corresponding entropy flux satisfy the following equations:
\begin{equation}
\label{eq:potential}
\varphi = \bm{W}\Tr \U - \mathcal{S},
\end{equation}
\begin{equation}
\label{eq:potential_flux}
 \psi_m = \bm{W}\Tr \Fxm - \mathcal{F}_{x_m}, \quad m=1, 2, 3,
\end{equation}
where the functions $\varphi$ and $\bm{\psi}$ are called the entropy potential and entropy potential flux, respectively.

We now show that the entropy inequality holds for the compressible Navier-Stokes equations in the time-dependent curvilinear coordinates.
Contracting Eq. (\ref{eq:NS_CC}) with the entropy variables given by Eq. (\ref{eq:Evariables}) yields
\begin{equation}
\label{eq:Eeq}
\begin{split}
& \overbrace{\bfnc{W}\Tr\frac{\partial J\U}{\partial \tau}}^{I} + \sum\limits_{l=1}^{3}\overbrace{\bfnc{W}\Tr\frac{\partial}{\partial \xi_l}\left( J\frac{\partial \xi_l}{\partial t}\U\right)}^{II}
+  \sum\limits_{m, l=1}^{3}\overbrace{\bfnc{W}\Tr\frac{\partial}{\partial \xi_l}\left(J\frac{\partial \xi_l}{\partial \xm}\Fxm\right)}^{III} =         \\
& \sum\limits_{l,n=1}^{3}\overbrace{\bfnc{W}\Tr\frac{\partial}{\partial\xil}\left(\Chatij{l}{n}\frac{\partial\bfnc{W}}{\partial \xi_n}\right)}^{IV}.
\end{split}
\end{equation}
The matrices $\Chatij{l}{n}$ on the right-hand side of Eq. (\ref{eq:Eeq}) are symmetric semi-definite matrices which are given by
\begin{equation}\label{eq:Chatij}
\Chatij{l}{n} \equiv J\frac{\partial \xi_l}{\partial x_n}\sum\limits_{m,j=1}^{3}\Cij{m}{j}.
\end{equation}
For further details on how $\Cij{m}{j}$ and $\Chatij{l}{n}$ are constructed, we refer the reader to \cite{Fisher2012phd}.

Using $\bfnc{W}\Tr = \frac{\partial \mathcal{S}}{\partial \U}$, the term $I$ in Eq. (\ref{eq:Eeq}) can be manipulated as follows:
\begin{equation}\label{eq:EeqI}
 I = J \frac{\partial \mathcal{S}}{\partial \U}\frac{\partial \U}{\partial \tau}+\bfnc{W}\Tr \U\frac{\partial J}{\partial \tau} =
\frac{\partial (J \mathcal{S})}{\partial \tau}+\left(\bfnc{W}\Tr\U - \mathcal{S}\right)\frac{\partial J}{\partial \tau}.
\end{equation}
Similarly, the term $II$ is recast in the following form:
\begin{equation}\label{eq:EeqII}
\begin{split}
&II =
\sum\limits_{l=1}^3 
J\frac{\partial \xi_l}{\partial t}\frac{\partial \mathcal{S}}{\partial \U}\frac{\partial \U}{\partial \xi_l} +
\bfnc{W}\Tr\U \frac{\partial}{\partial \xi_l}\left(J\frac{\partial \xi_l}{\partial t}\right) = \\
&\sum\limits_{l=1}^3 
\frac{\partial}{\partial \xi_l}\left( J\frac{\partial \xi_l}{\partial t}\mathcal{S}\right) +
\left(\bfnc{W}\Tr\U - \mathcal{S}\right) \frac{\partial}{\partial \xi_l}\left(J\frac{\partial \xi_l}{\partial t}\right).
\end{split}
\end{equation}
Using the compatibility relations (Eq. (\ref{eq:contract})), the term $III$ is reduced to
\begin{equation}\label{eq:EeqIII}
\begin{split}
& III = \sum\limits_{l,m=1}^3 
J\frac{\partial \xi_l}{\partial \xm}\frac{\partial \mathcal{S}}{\partial \U}\frac{\partial \Fxm}{\partial \xi_l} +
\bfnc{W}\Tr\Fxm \frac{\partial}{\partial \xi_l} \left(J\frac{\partial \xi_l}{\partial \xm}\right) =                                               \\
& \sum\limits_{l,m=1}^3
\frac{\partial}{\partial \xi_l}\left( J\frac{\partial \xi_l}{\partial \xm}\mathcal{F}_{x_m}\right) +
\sum\limits_{m=1}^3 \bfnc{W}\Tr\Fxm \sum\limits_{l=1}^3 \frac{\partial}{\partial \xi_l} \left(J\frac{\partial \xi_l}{\partial \xm}\right). 
\end{split}
\end{equation}
The last term in Eq. (\ref{eq:Eeq}) can be manipulated as follows:
\begin{equation}\label{eq:EeqIV}
IV =  \sum\limits_{l,n=1}^{3}\frac{\partial}{\partial\xil}\left(\bfnc{W}\Tr\Chatij{l}{n}\frac{\partial\bfnc{W}}{\partial \xi_n}\right) 
 - \frac{\partial\bfnc{W}}{\partial \xi_l}\Tr\Chatij{l}{n}\frac{\partial\bfnc{W}}{\partial \xi_n}.
\end{equation}

Integrating Eq. (\ref{eq:Eeq}) over the physical domain and taking into account Eqs. (\ref{eq:EeqI}--\ref{eq:EeqIV}), we have
\begin{equation}\label{eq:Eineq}
\begin{split}
&\int_{\Ohat}\left[ 
\frac{\partial (J \mathcal{S})}{\partial \tau}+ \sum\limits_{l=1}^3 \frac{\partial}{\partial \xi_l}\left( J\frac{\partial \xi_l}{\partial t}\mathcal{S}\right)
+\left(\bfnc{W}\Tr\U - \mathcal{S}\right)\left( \frac{\partial J}{\partial \tau} +  \sum\limits_{l=1}^3\frac{\partial}{\partial \xi_l}\left(J\frac{\partial \xi_l}{\partial t}\right)\right)\right]\mr{d}\Ohat \\
&+\int_{\Ohat}\left[
\sum\limits_{l,m=1}^3 \frac{\partial}{\partial \xi_l}\left( J\frac{\partial \xi_l}{\partial \xm}\mathcal{F}_{x_m}\right) +
\sum\limits_{m=1}^3 \bfnc{W}\Tr\Fxm \sum\limits_{l=1}^3 \frac{\partial}{\partial \xi_l} \left(J\frac{\partial \xi_l}{\partial \xm}\right)  
\right]\mr{d}\Ohat =\\
& +\int_{\Ohat}\left[
\sum\limits_{l,n=1}^{3}\frac{\partial}{\partial\xil}\left(\bfnc{W}\Tr\Chatij{l}{n}\frac{\partial\bfnc{W}}{\partial \xi_n}\right) 
 - \frac{\partial\bfnc{W}}{\partial \xi_l}\Tr\Chatij{l}{n}\frac{\partial\bfnc{W}}{\partial \xi_n}
\right]\mr{d}\Ohat.
\end{split}
\end{equation}
The last terms in the first two lines of Eq.(\ref{eq:Eineq}) are identically equal to zero because of the GCL equations given by Eq. (\ref{eq:GCL_C}).
Using the integration-by-parts (IBP) formula, the above equation can be recast in the following form:
\begin{equation}\label{eq:Eineq.2}
\begin{split}
\int_{\Ohat}\left[ 
\frac{\partial (J \mathcal{S})}{\partial \tau}+ \sum\limits_{l=1}^3 \frac{\partial}{\partial \xi_l}\left( J\frac{\partial \xi_l}{\partial t}\mathcal{S}\right)\right]\mr{d}\Ohat = & \sum\limits_{l,m=1}^{3}\oint_{\Gamma}\left(\bfnc{W}\Tr\Chatij{l}{m}\frac{\partial\bfnc{W}}{\partial \xi_m}-\Alm \mathcal{F}_{x_m}\right)\nxil\mr{d}\Ghat \\
& -\sum\limits_{l,n=1}^{3}\int_{\Ohat}\frac{\partial\bfnc{W}}{\partial \xi_l}\Tr\Chatij{l}{n}\frac{\partial\bfnc{W}}{\partial \xi_n}\mr{d}\Ohat. 
\end{split}
\end{equation}
Taking into account that the matrices $\Chatij{l}{m}$ are positive semi-definite and assuming that the boundary conditions are entropy stable, 
Eq. (\ref{eq:Eineq.2}) becomes
\begin{equation}\label{eq:Eineq.3}
\int_{\Omega}\frac{d \mathcal{S}}{d t} = 
\int_{\Ohat}\left[ \frac{\partial (J \mathcal{S})}{\partial \tau}+ \sum\limits_{l=1}^3 \frac{\partial}{\partial \xi_l}\left( J\frac{\partial \xi_l}{\partial t}\mathcal{S}\right)\right]\mr{d}\Ohat  \leq 0.
\end{equation} 
Equation (\ref{eq:Eineq.3}) represents the entropy inequality in the domain.
Note that for the Euler equations with smooth solutions, Eq. \eqref{eq:Eineq.2} becomes an equality.
The entropy inequality \eqref{eq:Eineq.2} is only a necessary condition,
which is not by itself sufficient to guarantee the convergence to a physically relevant 
weak solution of the Navier-Stokes equations.
%
% The skew-symmetric form of the compressible Navier-Stokes equations that is discretized is given as
% \begin{equation}\label{eq:NSCCS}
% \begin{split}
% &\frac{1}{2}\frac{\partial\Jk\Qk}{\partial \tau}+\frac{1}{2}\Jk\frac{\partial\Qk}{\partial\tau}
% +\frac{1}{2}\sum\limits_{l=1}^{3}\left[\frac{\partial}{\partial\xil}\left(\Blk\Qk\right)+\Blk\frac{\partial\Qk}{\partial\xil}\right]\\
% &+\frac{1}{2}\sum\limits_{l,m=1}^{3}\left[\frac{\partial }{\partial \xil}\left(\Almk\Fxmk\right)
% +\Almk\frac{\partial \Fxmk}{\partial \xil}\right]
% =\sum\limits_{l,a=1}^{3}\frac{\partial}{\partial\xil}\left(\Chatij{l}{a}\frac{\partial\bfnc{W}}{\partial \xia}\right).
% \end{split}
% \end{equation}

\section{Discrete analysis}\label{sec:discrete}
%============================

In this section, the various components of the entropy stable semi-discrete form for the 
compressible Navier-Stokes equations are presented. The approach taken is to first construct an entropy conservative semi-discrete form for the Euler equations, \cref{sec:Eulercon}, 
then add appropriate interface dissipation to make the scheme entropy stable, \cref{sec:Eulerstab}, and finally, 
discretize the viscous components such that entropy stability is not lost, \cref{sec:NS}. 
% In addition, a brief review of the weak imposition of boundary conditions is given in \cref{sec:BC}.

\subsection{Entropy conservative semi-discrete form for the Euler equations}\label{sec:Eulercon}
%-----------------------------------------------------------------------------------------------------------------------
The skew-symmetric splitting of the Euler equations that is discretized on the $\kappa^{\rm{th}}$ element is given as
\begin{equation}\label{eq:skewEuler}
\begin{split}
\frac{1}{2}\frac{\partial\Jk\Qk}{\partial\tau}+\frac{1}{2}\Jk\frac{\partial\Qk}{\partial\tau}
&+\frac{1}{2}\sum\limits_{l=1}^{3}\left[
\frac{\partial}{\partial\xil}\left(\Blk\Qk\right)+\Blk\frac{\partial\Qk}{\partial \xil}
\right]\\
&+\frac{1}{2}\sum\limits_{l,m=1}^{3}\left[
\frac{\partial}{\partial\xil}\left(\Almk\Fxm\right)+\Almk\frac{\partial\Fxm}{\partial \xil}
\right] = \bm{0}.
\end{split}
\end{equation}
The skew--symmetric form is derived by averaging the conservative and nonconservative forms.
Note that the conservative form is derived by taking advantage of  the GCL conditions outlined in \cref{sec:GCL}:
\begin{equation}\label{eq:GCLtime}
\frac{\partial\Jk}{\partial\tau}+\sum\limits\frac{\partial}{\partial\xil}\left(\Blk\right)=0,
\end{equation}
\begin{equation}\label{eq:GCLspace}
\sum\limits_{l=1}^{3}\frac{\partial}{\partial\xil}\left(\Almk\right)=0\qquad m = 1,2,3.
\end{equation}

In order to achieve entropy conservation, the approximations to the spatial derivatives are constructed such that the 
continuous entropy stability analysis is mimicked in a one-to-one fashion. To do so, 
the approximation to the spatial derivatives needs to be mimetic of a special case of integration by parts (IBP), namely 
\begin{equation}\label{eq:nonIBP}
\int\bfnc{V}\Tr\frac{\partial\bfnc{G}}{\partial\zeta}\mr{d}D=
\int_{D}\frac{\partial\fnc{h}}{\partial\zeta}\mr{d}D = 
\oint_{\Gamma}\fnc{H}n_{\zeta}\mr{d}\Gamma,
\end{equation}
(this is referred to as a nonlinear IBP relation) where $\bfnc{V}$ and $\bfnc{G}\left(\bfnc{U}\right)$ are vector valued functions, 
$\fnc{h}\left(\bfnc{U}\right)$ is a scalar function, $\zeta$ is some independent 
variable over the domain $D$ with boundary $\Gamma$, and $n_{\zeta}$ is the component of the outward facing unit normal in the $\zeta$ direction.
\begin{remark}
In the current context, the two nonlinear IBP formulas that are of interest are 
\begin{equation*}
\begin{split}
&\frac{1}{2}\int_{\Ohat}\bfnc{W}\Tr\left\{
\frac{\partial\left( J\U\right)}{\partial\tau}+J\frac{\partial\U}{\partial \tau}
+\sum\limits_{l=1}^{3}\left[\frac{\partial}{\partial\xil}\left(\Bl\bfnc{U}\right)
+\Bl\frac{\partial\bfnc{U}}{\partial\xil}
\right]\right\}\mr{d}\Ohat=\\
&\int_{\Ohat}\frac{\partial J \mathcal{S}}{\partial \xil}\mr{d}\Ohat+\sum\limits_{l=1}^{3}\oint_{\Ghat}\mathcal{S}\Bl\nxil\mr{d}\Ghat,\\
&\frac{1}{2}\sum\limits_{l=1}^{3}\int_{\Ohat}\bfnc{W}\Tr\left[
\frac{\partial}{\partial\xil}\left(\Alm\Fxm\right)
+\Alm\frac{\partial\Fxm}{\partial\xil}
\right]
\mr{d}\Ohat=
\sum\limits_{l=1}^{3}\oint_{\Ghat}\mathcal{F}_{x_m}\Alm\nxil\mr{d}\Ghat,\\
&m=1,2,3,
\end{split}
\end{equation*}
where the GCL conditions~\eqref{eq:GCLtime} and~\eqref{eq:GCLspace} have been used.
\end{remark}
To introduce nonlinear SBP operators and their associated notation, the generic nonlinear
IBP property~\eqref{eq:nonIBP} and the derivative $\frac{\partial\bfnc{G}\left(\bfnc{U}\right)}{\partial \zeta}$ are used; the derivative 
is approximated using the following nonlinear SBP operator: 
\begin{equation}\label{eq:nonlinearSBP}
\left.\frac{\partial\bfnc{G}\left(\bfnc{U}\right)}{\partial \zeta}\right|_{\Ck}\approx
2\Dzeta\circ\matGsc{\qk}{\qk}\one.
\end{equation}
The symbol $\circ$ is the Hadamard product, i.e., element-wise multiplication while $\Ck$ is the set of nodes on the $\kappa^{\rm{th}}$ element; therefore, the notation 
\[
\left.\frac{\partial\bfnc{G}\left(\bfnc{U}\right)}{\partial \zeta}\right|_{\Ck}
\]
denotes the vector constructed by evaluating $\frac{\partial\bfnc{G}\left(\bfnc{U}\right)}{\partial \zeta}$ at the nodes $\Ck$. 
The vector $\qk$ is a vector of vectors corresponding to the discrete solution of the $5$ 
conservation equations at each node ordered with the $\xithree-$direction varying most rapidly, 
then $\xitwo$, and then $\xione$. The vector $\one$ is defined as $\one\equiv\bm{1}\otimes\onefive$, where 
$\bm{1}$ is a vector of ones of size $N$, the number of nodes in each element, and $\onefive$ 
is a vector of ones of size $5$. It is necessary to discuss the components of $\qk$ at each node and 
the values of conservative variables at all nodes in an element. For this purpose a matrix like notation, $\qkmatrix{i}{j}$, 
is introduced to index the vector such that $\qkmatrix{i}{:} = \qk((i-1)5:i5)$ and $\qkmatrix{:}{j}$ is the vector constructed 
from the $j^{\rm{th}}$ conserved variable at all nodes. Furthermore, the two-argument matrix function $\matGsc{\qk}{\qk}$ is 
a block matrix of diagonal blocks, that is 
\begin{equation*}
\begin{split}
&\matGsc{\qk}{\qr} \equiv \\
&\left[
\begin{array}{ccc}
\diag(\Gsc{\qkmatrix{1}{:}}{\qrmatrix{1}{:}})&\dots&\diag(\Gsc{\qkmatrix{1}{:}}{\qrmatrix{N}{:}})\\
\vdots&&\vdots\\
\diag(\Gsc{\qkmatrix{N}{:}}{\qrmatrix{1}{:}})&\dots&\diag(\Gsc{\qkmatrix{N}{:}}{\qrmatrix{N}{:}})
\end{array}
\right],
\end{split}
\end{equation*}
and it is of size $\left(5\,N\right)\times\left(5\,N\right)$. The five element column vectors, 
\[
\Gsc{\qkmatrix{i}{:}}{\qrmatrix{j}{:}},
\]
are constructed from two-point flux functions that are symmetric in their arguments and consistent, thus,
\begin{equation*}
\begin{split}
&\Gsc{\qkmatrix{i}{:}}{\qrmatrix{j}{:}} = \Gsc{\qrmatrix{j}{:}}{\qkmatrix{i}{:}},\\
&\Gsc{\qkmatrix{i}{:}}{\qkmatrix{i}{:}} = \bfnc{G}\left(\qkmatrix{i}{:}\right),
\end{split}
\end{equation*}
where the superscript denotes that this is an entropy conservative two-point flux function. This implies that 
\begin{equation}\label{eq:symmf}
\matGsc{\qk}{\qk}=\matGsc{\qk}{\qk}\Tr,\quad\matGsc{\qk}{\qr}=\matGsc{\qr}{\qk}\Tr.
\end{equation}
Finally, in order to construct entropy conservative schemes, the two-point flux 
functions must satisfy the Tadmor shuffle condition~\cite{Tadmor2003}:
\begin{equation}\label{eq:shuffle}
\begin{split}
&\left(\vkmatrix{i}{:}-\vrmatrix{j}{:}\right)\Tr\Gsc{\qkmatrix{i}{:}}{\qrmatrix{j}{:}}=\\
&\left(\vkmatrix{i}{:}\Tr\gkmatrix{i}{:}-\hk(i)\right)
-\left(\vrmatrix{j}{:}\Tr\grmatrix{j}{:}-\hr(j)\right).
\end{split}
\end{equation}

The next theorem demonstrates 
that the nonlinear SBP operator is mimetic of the nonlinear IBP 
relation~\eqref{eq:nonIBP} to high order~\cite{Fernandez2018prefine}:
\begin{theorem}\label{thrm:nonIBP}
Consider a two-argument matrix flux 
function $\matGsc{\qk}{\qr}$ constructed from the two point 
flux function $\Gsc{\qkmatrix{i}{:}}{\qrmatrix{j}{:}}$ that 
 is symmetric, consistent, and satisfies the Tadmor shuffle 
condition~\eqref{eq:shuffle}. Then the nonlinear SBP operator 
is an approximation to the derivative as below,
\[
2\Dzeta\circ\matGsc{\qk}{\qk}\one\approx\frac{\partial\bfnc{G}}{\partial\zeta},
\]
and satisfies the following nonlinear SBP property,
\begin{equation}\label{eq:nonSBP}
\begin{split}
\overbrace{2\vk\Tr\M\Dzeta\circ\matGsc{\qk}{\qk}\one}
^{\displaystyle\approx\int_{D}\bfnc{V}\frac{\partial\bfnc{G}}{\partial\zeta}\mr{d}D}
& =
\overbrace{
-\onesca\Tr\Ezetasca\left(\vk\Tr\gk-\hk\right) 
+\vk\Tr\Ezeta\circ\matGsc{\qk}{\qk}\one}
^{\displaystyle\approx\oint_{\Gamma}\fnc{h}n_{\zeta}\mr{d}\Gamma},
\end{split}
\end{equation}
where the LHS and RHS are high-order approximations to the LHS and RHS of 
the nonlinear IBP property~\eqref{eq:nonIBP}.
\end{theorem}
\begin{proof}
The proof is given in~\cite{Fernandez2018prefine}.
\end{proof}
Theorem~\ref{thrm:nonIBP} is proven using accuracy estimates on the nonlinear SBP operators and the bilinear forms that result 
from its constituent components. In addition to Theorem~\ref{thrm:nonIBP}, the following theorem (which will be used to prove entropy conservation) 
is employed:
\begin{theorem}\label{thrm:telescope}
Consider the matrix $\mat{A}$ of size $N\times N$ and a two-argument matrix flux 
function $\matGsc{\qk}{\qr}$ constructed from the two point 
flux function $\Gsc{\bm{u}}{\bm{v}}$ that satisfies the Tadmor shuffle condition~\eqref{eq:shuffle}
and is symmetric, then
\begin{equation*}
\begin{split}
&\vk\Tr\left[\left(\overline{\mat{A}}\right)\circ\matGsc{\qk}{\qr}\right]\one-\one\Tr\left[\left(\overline{\mat{A}}\right)\circ\matGsc{\qk}{\qr}\right]\vr =\\
&\left(\vk\Tr\gk-\hk\right)\Tr\mat{A}\onesca-\onesca\Tr\mat{A}\left(\vr\Tr\gr-\hr\right),
\end{split}
\end{equation*}
where $\bm{v}\Tr\bm{g}$ is to be understood as $\left(\bm{v}\Tr\bm{g}\right)(i)\equiv\bm{v}(: , \,i)\Tr\bm{g}(: , \,i)$.
\end{theorem}
\begin{proof}
The proof is given in~\cite{Fernandez2018prefine}.
\end{proof}

The discrete operators that are used to discretize the terms in~\eqref{eq:skewEuler} and their error properties are 
given in the next theorem.
\begin{theorem}\label{thrm:AlmkDxil}
\begin{equation}\label{eq:AlmkDxil}
\left(\matAlmk\Dxil\right)\circ\matFxmsc{\qk}{\qk}\one\approx\left.\frac{1}{2}\Alm\frac{\partial\Fxm}{\partial\xil}\right|_{\Ck},
\end{equation}
\begin{multline}\label{eq:DxilAlmk}
\left(\Dxil\matAlmk\right)\circ\matFxmsc{\qk}{\qk}\one\approx\\\left.\frac{1}{2}\frac{\partial}{\xil}\left(\Alm\Fxm\right)\right|_{\Ck}
+\left.\frac{1}{2}\Fxm\frac{\partial}{\partial\xil}\left(\Alm\right)\right|_{\Ck},
\end{multline}

\begin{equation}\label{eq:BlkDxil}
\left(\matBlk\Dxil\right)\circ\matQsc{\qk}{\qk}\one\approx\left.\frac{1}{2}\Bl\frac{\partial\Qk}{\partial\xil}\right|_{\Ck},
\end{equation}
\begin{equation}\label{eq:DxilBlk}
\left(\Dxil\matBlk\right)\circ\matQsc{\qk}{\qk}\one\approx\left.\frac{1}{2}\frac{\partial}{\xil}\left(\Bl\Qk\right)\right|_{\Ck}
+\left.\frac{1}{2}\Qk\frac{\partial}{\partial\xil}\left(\Bl\right)\right|_{\Ck},
\end{equation}
where $\matAlmk$ and $\matBlk$ are diagonal matrices with the metric terms along their diagonals and $\mat{F}_{\xm}^{sc}$ and $\mat{U}^{sc}$ 
are constructed from two-point flux functions that are symmetric, consistent, and satisfy the following shuffle conditions:
\begin{equation}\label{eq:Uflux}
\begin{split}
&\left(\wkmatrix{i}{:}-\wrmatrix{j}{:}\right)\Tr\fxmsc{\qkmatrix{i}{:}}{\qrmatrix{j}{:}}=\\
&\wkmatrix{i}{:}\Tr\bfnc{F}_{\xm}(\qkmatrix{i}{:})-\mathcal{F}_{x_m}\left(\qkmatrix{i}{:}\right)\\
&-\left[\wrmatrix{j}{:}\Tr\bfnc{F}_{\xm}(\qrmatrix{j}{:})-\mathcal{F}_{x_m}\left(\qrmatrix{j}{:}\right)\right],
\end{split}
\end{equation} 
\begin{equation}\label{eq:Fflux}
\begin{split}
&\left(\wkmatrix{i}{:}-\wrmatrix{j}{:}\right)\Tr\Qsc{\qkmatrix{i}{:}}{\qrmatrix{j}{:}}=\\
&\wkmatrix{i}{:}\Tr\qkmatrix{i}{:}-\mathcal{S}\left(\qkmatrix{i}{:}\right)\\
&-\left[\wrmatrix{j}{:}\Tr\qrmatrix{j}{:}-\mathcal{S}\left(\qrmatrix{j}{:}\right)\right].
\end{split}
\end{equation} 
\end{theorem}
\begin{proof}
The proofs of~\eqref{eq:AlmkDxil} and~\eqref{eq:DxilAlmk} are given in~\cite{Fernandez2018prefine}, while the 
proofs of~\eqref{eq:BlkDxil} and~\eqref{eq:DxilBlk}, follow identically.
\end{proof}
Using the operators in Theorem~\ref{thrm:AlmkDxil}, the semi-discrete form of~\eqref{eq:skewEuler} is given as
\begin{multline}\label{eq:skewEulersemi}
\frac{\mr{d}}{\mr{d}\tau}\matJk\qk+\sum\limits_{l=1}^{3}\left(\Dxil\matBlk+\matBlk\Dxil\right)\circ\matQsc{\qk}{\qk}\one\\
+\sum\limits_{l,m=1}^{3}\left(\Dxil\matAlmk+\matAlmk\Dxil\right)\circ\matFxmsc{\qk}{\qk}\one=\bm{SAT}_{\tau}+\bm{SAT}_{\xil},
\end{multline}
where the SATs on the right-hand side are defined as
\begin{multline}\label{eq:SATcentraltau}
\bm{SAT}_{\tau}\equiv\M^{-1}\sum\limits_{l=1}^{3}\left[\left(\Exil\matBlk\right)\circ\matQsc{\qk}{\qk}\one\right.\\
+\left(\matBlk\Rxilalpha\Tr\Pxilortho\Rxilbeta\right)\circ\matQsc{\qk}{\qtwolmone}\one\\
\left.-\left(\matBlk\Rxilbeta\Tr\Pxilortho\Rxilalpha\right)\circ\matQsc{\qk}{\qtwol}\one\right],
\end{multline}
\begin{multline}\label{eq:SATcentralxil}
\bm{SAT}_{\xil}\equiv\M^{-1}\sum\limits_{l,m=1}^{3}\left[\left(\Exil\matAlmk\right)\circ\matFxmsc{\qk}{\qk}\one\right.\\
+\left(\matAlmk\Rxilalpha\Tr\Pxilortho\Rxilbeta\right)\circ\matQsc{\qk}{\qtwolmone}\one\\
\left.-\left(\matAlmk\Rxilbeta\Tr\Pxilortho\Rxilalpha\right)\circ\matFxmsc{\qk}{\qtwol}\one\right],
\end{multline}
where $\qtwolmone$ and $\qtwol$ are the solution vectors for the elements touching the $(2l-1)$ and $2l$ faces, 
respectively.

For analysis, it is convenient to separate out the discretization of each scalar conservation law 
in~\eqref{eq:skewEuler} and for this purpose the following is introduced: 
\begin{multline}\label{eq:skewEulersemisca}
\frac{\mr{d}}{\mr{d}\tau}\matJksca\qkmatrix{:}{i}+\sum\limits_{l=1}^{3}\left(\Dxilsca\matBlksca+\matBlksca\Dxilsca\right)\circ
\matQscsca{i}{\qk}{\qk}\onesca\\
+\sum\limits_{l,m=1}^{3}\left(\Dxilsca\matAlmksca+\matAlmksca\Dxilsca\right)\circ\matFxmscsca{i}{\qk}{\qk}\onesca=\bm{SAT}_{\tau}^{(i)}+\bm{SAT}_{\xil}^{(i)},
\end{multline}
where the SATs on the right-hand side are defined as
\begin{multline}\label{eq:SATcentraltausca}
\bm{SAT}_{\tau}^{(i)}\equiv\Msca^{-1}\sum\limits_{l=1}^{3}\left[\left(\Exilsca\matBlksca\right)\circ\matQscsca{i}{\qk}{\qk}\onesca\right.\\
+\left(\matBlksca\Rxilalphasca\Tr\Pxilorthosca\Rxilbetasca\right)\circ\matQscsca{i}{\qk}{\qtwolmone}\onesca\\
\left.-\left(\matBlksca\Rxilbetasca\Tr\Pxilorthosca\Rxilalphasca\right)\circ\matQscsca{i}{\qk}{\qtwol}\onesca\right],
\end{multline}
\begin{multline}\label{eq:SATcentralxilsca}
\bm{SAT}_{\xil}^{(i)}\equiv\Msca^{-1}\sum\limits_{l,m=1}^{3}\left[\left(\Exilsca\matAlmksca\right)\circ\matFxmscsca{i}{\qk}{\qk}\onesca\right.\\
+\left(\matAlmksca\Rxilalphasca\Tr\Pxilorthosca\Rxilbetasca\right)\circ\matQscsca{i}{\qk}{\qtwolmone}\onesca\\
\left.-\left(\matAlmksca\Rxilbetasca\Tr\Pxilorthosca\Rxilalphasca\right)\circ\matFxmscsca{i}{\qk}{\qtwol}\onesca\right],
\end{multline}
where
\begin{equation*}
\begin{split}
&\matQscsca{i}{\qk}{\qr} \equiv\\
&\left[
\begin{array}{ccc}
\Qsc{\qkmatrix{1}{:}}{\qrmatrix{1}{:}}(i)&\dots&\Qsc{\qkmatrix{1}{:}}{\qrmatrix{N_{r}}{:}}(i)\\
\vdots&&\vdots\\
\Qsc{\qkmatrix{N_{\kappa}}{:}}{\qrmatrix{1}{:}}(i)&\dots&\Qsc{\qkmatrix{N_{\kappa}}{:}}{\qrmatrix{N_{r}}{:}}(i)
\end{array}
\right],\\
&i=1,\dots,5.
\end{split}
\end{equation*}

Similarly,
\begin{equation*}
\begin{split}
&\matFxmscsca{i}{\qk}{\qr} \equiv\\
&\left[
\begin{array}{ccc}
\fxmsc{\qkmatrix{1}{:}}{\qrmatrix{1}{:}}(i)&\dots&\fxmsc{\qkmatrix{1}{:}}{\qrmatrix{N_{r}}{:}}(i)\\
\vdots&&\vdots\\
\fxmsc{\qkmatrix{N_{\kappa}}{:}}{\qrmatrix{1}{:}}(i)&\dots&\fxmsc{\qkmatrix{N_{\kappa}}{:}}{\qrmatrix{N_{r}}{:}}(i)
\end{array}
\right],\\
&i=1,\dots,5.
\end{split}
\end{equation*}
In order to be entropy conservative, the approximation to the metric terms used in~\eqref{eq:skewEulersemi} need 
to satisfy the conditions contained in the next theorem.
\begin{theorem}\label{thrm:EulerScon}
The discretization~\eqref{eq:skewEulersemi} has a telescoping entropy on interior elements if 
the metric terms satisfy
\begin{equation}\label{eq:GCLtaudiss}
\frac{\mr{d}\vecJk}{\mr{d}\tau}+\sum\limits_{l=1}^{3}\Dxilsca\matBlksca\onesca = \bm{0},
\end{equation}
and
\begin{equation}\label{eq:GCLxildiss}
\sum\limits_{l=1}^{3}\Dxilsca\matAlmksca\onesca = \bm{0},\qquad m = 1,2,3.
\end{equation}
This implies that for the periodic problem, the scheme is entropy conservative, i.e., 
\begin{equation*}
\begin{split}
&\sum\limits_{\kappa=1}^{K}\onesca\Tr\Msca\frac{\mr{d}\matJksca\mathcal{S}_k}{\mr{d}\tau} = 0,
\end{split}
\end{equation*}
and entropy conservative/stable for non-periodic problems under the assumption of appropriate 
SATs are used for the weak imposition of boundary conditions.
\end{theorem}
\begin{proof}
The proof is contained in \cref{app:Eulerentropycon}
\end{proof}

Next, let us show that the proposed scheme is freestream preserving.
\begin{theorem}\label{thrm:EulerFS}
Under the same conditions on the metric terms as in Theorem~\ref{thrm:EulerScon} the discretization~\eqref{eq:skewEulersemi} 
is freestream preserving.
\end{theorem}
\begin{proof}
The proof follows similarly to that given in Ref.~\cite{Fernandez2018prefine}.
Here, we sketch the required steps. It is convenient to use the scalar version of the semi-discrete form~\eqref{eq:skewEulersemisca}.
Consider a physically meaningful constant state, $\U_{c}$, then we denote the evaluation of this constant state on the mesh nodes of the 
$\kappa^{\rm{th}}$ element as $\qk^{(c)}$; then, via the consistency of the two-point flux functions, the derivative operators 
in the scalar version of semi-discrete form reduce to
\[
\left(\Dxilsca\matBlksca+\matBlksca\Dxilsca\right)\circ
\matQscsca{i}{\qk^{(c)}}{\qk^{(c)}}\onesca = \U_{c}(i)\Dxilsca\matBlksca\onesca
\]
and
\begin{multline*}
\left(\Dxilsca\matAlmksca+\matAlmksca\Dxilsca\right)\circ\matFxmscsca{i}{\qk^{(c)}}{\qk^{(c)}}\onesca=\\
\Fxm\left(\U_{c}\right)(i)\Dxilsca\matAlmksca\onesca.
\end{multline*}
In a similar manner, it is easy to show that the SATs vanish and the resulting conditions on the 
metric terms are the same as in Theorem~\ref{thrm:EulerScon}.
%For our own purposes I include the proof in \cref{app:FS}.
\end{proof}
\begin{remark}
There is a misconception that the freestream preservation is sufficient to guarantee stability on moving and deforming grids.
In fact, the freestream preservation property is only a necessary condition for stability.
\end{remark}
In the present analysis, the interest is in schemes that are discretely conservative so that when they converge, they converge to a weak solution of 
the conservation law equations. 
A proof that the scheme has the desired property is based on a semi-discrete analog
%variation 
of the Lax-Wendroff theorem. The conservation property can be demonstrated by showing 
%and the idea is to demonstrate 
that the scheme has a telescoping flux form. 
Because the nonlinear SBP operators are used in the construction of the semi-discrete form~\eqref{eq:skewEulersemi}, a 
proof of a ``stronger'' discrete conservation statement, namely subcell conservation,  is  given herein. 
\begin{theorem}\label{thrm:subcell}
The semi-discrete scheme~\eqref{eq:skewEulersemi} is subcell conservative.
\end{theorem}
\begin{proof}
The proof is sketched here. The nonlinear SBP operators can be recast in telescoping flux form. For example, 
for the one-dimensional nonlinear operator on $N$ nodes
\begin{equation*}
2\left(\DxiloneD\otimes\mat{I}_{3}\right)\circ\matFxmsc{\qk}{\qk}\one = 2\mat{P}^{-1}\otimes\mat{I}_{3}\mat{\Delta}\otimes\mat{I}_{3}\bfnc{f}_{\xm},
\end{equation*}
where the rectangular matrix $\mat{\Delta}$ of size $(N+1)\times N$ is given as
\begin{equation*}
\mat{\Delta} \equiv
\left[
\begin{array}{ccccc}
-1&1\\
&-1&1\\
&&\ddots&\ddots\\
&&&-1&1
\end{array}
\right]. 
\end{equation*}
The vector $\bfnc{f}_{\xm}$ is an interpolation of $\fnc{F}_{\xm}$ to a set of nodes that interdigitate the nodal distribution of the 
element. The nonlinear SBP operators constructed globally over the mesh are then recast in a telescoping flux form over each mesh line. 
The resulting finite-volume like scheme is in the form required by 
the Lax-Wendroff theorem and therefore, the discretization is subcell conservative (for more details, see Refs.~\cite{Fisher2013,Fisher2012phd}).
\end{proof}

\subsection{Entropy stable semi-discrete form for the Euler equations: interface dissipation}\label{sec:Eulerstab}
%------------------------------------------------------------------------------------------------------------------------------------------
The weak imposition of boundary conditions using SATs can be reinterpreted in the sense of discontinuous Galerkin methods 
by discretely integrating against a test function $\fnc{V}$; thus, 
\begin{equation*}
\begin{split}
&\vk\Tr\M\bm{SAT}_{\tau}\approx\oint_{\Ghatk}\sum\limits_{l=1}^{3}\fnc{V}\left(\bfnc{U}-\bfnc{U}^{*}\right)\Blk\nxil\mr{d}\Ghat,\\
&\vk\Tr\M\bm{SAT}_{\xil}\approx\oint_{\Ghatk}\sum\limits_{l,m=1}^{3}\fnc{V}\left(\bfnc{F}_{xm}-\bfnc{F}_{\xm}^{*}\right)\Almk\nxil\mr{d}\Ghat.
\end{split}
\end{equation*}
In the context of the entropy conserving SATs thus far introduced, for example on the $2l$ interface, 
the weak form of the SAT can be reinterpreted as
\begin{equation*}
\begin{split}
&\vk\Tr\bm{SAT}_{\tau}^{(2l)}=\sum\limits_{l=1}^{3}\overbrace{\vk\Tr\left(\Rxilbeta\Tr\Pxilortho\Rxilbeta\matBlk\right)\circ\matQsc{\qk}{\qk}\one}^{\approx\displaystyle\oint_{\Ghatk^{2l}}\sum\limits_{l=1}^{3}\fnc{V}\bfnc{U}\Blk\nxil\mr{d}\Ghat}\\
&-\overbrace{\vk\Tr\left(\matBlk\Rxilbeta\Tr\Pxilortho\Rxilalpha\right)\circ\matQsc{\qk}{\qtwol}\one}^{\approx\displaystyle\oint_{\Ghatk^{2l}}\sum\limits_{l=1}^{3}\fnc{V}\bfnc{U}^{*}\Blk\nxil\mr{d}\Ghat},
\end{split}
\end{equation*}
and
\begin{equation*}
\begin{split}
&\vk\Tr\bm{SAT}_{\xil}^{(2l)}=\sum\limits_{l,m=1}^{3}\overbrace{\vk\Tr\left(\Rxilbeta\Tr\Pxilortho\Rxilbeta\matAlmk\right)\circ\matFxmsc{\qk}{\qk}\one}^{\approx\displaystyle\oint_{\Ghatk^{2l}}\sum\limits_{l=1}^{3}\fnc{V}\bfnc{F}_{\xm}\Almk\nxil\mr{d}\Ghat}\\
&-\overbrace{\vk\Tr\left(\matAlmk\Rxilbeta\Tr\Pxilortho\Rxilalpha\right)\circ\matFxmsc{\qk}{\qtwol}\one}^{\approx\displaystyle\oint_{\Ghatk^{2l}}\sum\limits_{l=1}^{3}\fnc{V}\bfnc{F}_{\xm}^{*}\Blk\nxil\mr{d}\Ghat}.
\end{split}
\end{equation*}

In this section, the entropy conservative SATs are augmented with dissipative terms motivated by the upwinding used in the Roe approximate Riemann solver, which has generic form
\begin{equation*}
\bfnc{F}^{*}=\frac{\bfnc{F}^{+}+\bfnc{F}^{-}}{2}-\frac{1}{2}\mat{Y}|\mat{\Lambda}|\mat{Y}^{-1}\left(\bfnc{U}^{+}-\bfnc{U}^{+}\right),
\end{equation*}
where the superscript $+$ denotes quantities evaluated on a side of an interface in the positive direction of the unit normal, and the superscript $-$ denotes 
quantities in the negative direction. Since the Roe flux is not entropy consistent, the central flux,
\begin{equation*}
\frac{\bfnc{F}^{+}+\bfnc{F}^{-}}{2}
\end{equation*}
is replaced by the numerical fluxes given in the above entropy conservative SAT terms. To provide entropy dissipation at element interfaces, 
the dissipative term, i.e.,
\begin{equation*}
\mat{Y}|\mat{\Lambda}|\mat{Y}^{-1}\left(\bfnc{U}^{+}-\bfnc{U}^{+}\right).
\end{equation*}
should also be modified.

The approach is to construct a dissipation term that enables entropy stability, while remaining design order accurate and conservative. 
This is accomplished by using the flux Jacobian with respect to the entropy variables rather than the conservative variables. The 
eigenvectors of the conservative variable flux Jacobian can be scaled such that 
\begin{equation*}
\frac{\partial \bfnc{U}}{\partial\bfnc{W}}=\mat{Y}\mat{Y}\Tr.
\end{equation*}
Thus, for the mesh velocity terms such as $\Blk\frac{\partial\bfnc{U}}{\partial \xil}$, the added dissipative term is of the form
\begin{equation*}
\left|\Blk\right|\mat{Y}\mat{Y}\Tr\left(\bfnc{W}^{+}-\bfnc{W}^{-}\right).
\end{equation*}
For the remaining spatial terms, note that
\begin{equation*}
\frac{\partial}{\partial\bfnc{W}}\left(\sum\limits_{m=1}^{3}\Almk\bfnc{F}_{\xm}\right)=\frac{\partial\bfnc{F}_{\xil}}{\partial\bfnc{W}}=
\frac{\partial\bfnc{F}_{\xil}}{\partial\bfnc{U}}\frac{\partial\bfnc{U}}{\partial\bfnc{W}}=\mat{Y}|\mat{\Lambda}_{\xil}|\mat{Y}\Tr,
\end{equation*}
this property was first derived by Merriam~\cite{Merriam1989} (for more details, see \cite{Fisher2012phd}).
Thus, the added dissipative terms are
\begin{equation}\label{eq:diss}
\begin{split}
\bm{diss}_{\tau}\equiv&-\M^{-1}\Rxilalpha\Tr\Pxilortho\left|\Blk\frac{\partial\bfnc{U}}{\partial\bfnc{W}}\right|_{(2l-1)}\left(\Rxilalpha\wk-\Rxilbeta\wtwolmone\right)\\
&-\M^{-1}\Rxilbeta\Tr\Pxilortho\left|\Blk\frac{\partial\bfnc{U}}{\partial\bfnc{W}}\right|_{2l}\left(\Rxilbeta\wk-\Rxilalpha\wtwol\right),\\
\bm{diss}_{\xil}\equiv&-\M^{-1}\Rxilalpha\Tr\Pxilortho\left|\frac{\partial\bfnc{F}_{\xil}}{\partial\bfnc{W}}\right|_{(2l-1)}\left(\Rxilalpha\wk-\Rxilbeta\wtwolmone\right)\\
&-\M^{-1}\Rxilbeta\Tr\Pxilortho\left|\frac{\partial\bfnc{F}_{\xil}}{\partial\bfnc{W}}\right|_{2l}\left(\Rxilbeta\wk-\Rxilalpha\wtwol\right),
\end{split}
\end{equation}
where
\begin{equation*}
\begin{split}
&\left|\Blk\frac{\partial\bfnc{U}}{\partial\bfnc{W}}\right|\equiv\left|\Blk\right|\mat{Y}\mat{Y}\Tr,\\
&\left|\frac{\partial\bfnc{F}_{\xil}}{\partial\bfnc{W}}\right|\equiv\mat{Y}|\mat{\Lambda}_{\xil}|\mat{Y}\Tr
\end{split} 
\end{equation*}
and the matrices $\mat{Y}$ and $\mat{\Lambda}_{\xil}$ are constructed pointwise from the Roe-average of the two states on either side of the interface.
The next theorem characterizes the addition of the dissipation to the scheme.
\begin{theorem}\label{thrm:diss}
The added dissipative terms~\eqref{eq:diss} are design order, lead to an entropy stable scheme, and retain subcell conservation.
\end{theorem}
\begin{proof}
%The proof follows similarly to that given in Ref.~\cite{Fernandez2018prefine}.
The proofs of design order and entropy stability follow similarly to those given in Ref.~\cite{Fernandez2018prefine}; 
subcell conservation follows by recasting the dissipation terms in telescoping flux form.
%For our own purposes I include the proof in \cref{app:dissEuler}.
\end{proof}

\subsection{Extension to the Navier-Stokes equations}\label{sec:NS}
%-----------------------------------------------------------------------------------
A local discontinuous Ga-lerkin (LDG)-type approach is used to discretize the viscous terms
(see~\cite{Carpenter2014,Carpenter2015,Parsani2015b,Parsani2016}).
Thus, Eq.\eqref{eq:NS_CC} is discretized as follows:
\begin{equation}\label{eq:NSD}
\begin{split}
&\frac{1}{2}\frac{\mr{d}}{\mr{d}t}\left(\matJk\qk\right)+\frac{1}{2}\matJk\frac{\mr{d}\qk}{\mr{d}t}+\FI-\SATI =
\sum\limits_{l,n=1}^{3}\Dxil\matChatijk{l}{n}{k}\Thetaak{n}{\kappa}\\
&-\frac{1}{2}\M^{-1}\sum\limits_{l,n=1}^{3}
\left[\Exil\matChatijk{l}{n}{k}\Thetaak{n}{\kappa}
-\Rxilbeta\Tr\Pxilortho\Rxilalpha\matChatijk{l}{n}{2l}\Thetaak{n}{2l}\right.\\
&\left.-\Rxilalpha\Tr\Pxilortho\Rxilbeta\matChatijk{l}{n}{(2l-1)}\Thetaak{n}{(2l-1)}\right]
+\IPk,
\end{split}
\end{equation}
where 
\begin{equation}\label{eq:Theta}
\begin{split}
& \Thetaak{n}{\kappa}=\Dxia{n}\wk \\
& -\frac{1}{2}\M^{-1}\left(\Exia{n}\wk-\Rxiabeta{n}\Tr\Pxiaortho{n}\Rxiaalpha{n}\wtwoa{n}-\Rxiaalpha{n}\Tr\Pxiaortho{n}\Rxiabeta{n}\wtwoamone{n}\right).
\end{split}
\end{equation}
The terms $\FI$ and $\SATI$ are the discretization of the inviscid portion of the governing equations. The interior penalty term, $\IPk$, 
adds interface dissipation and is given as
\begin{equation}\label{eq:IP}
\begin{split}
\IPk\equiv&-\M^{-1}\sum\limits_{l=1}^{3}\left(\matJk^{-1}\matChatijk{l}{l}{\kappa,(2l-1)}\right)\Rxilalpha\Tr\Pxilortho\left(\Rxilalpha\wk-\Rxilbeta\wtwolmone\right)\\
&-\M^{-1}\sum\limits_{l=1}^{3}\left(\matJk^{-1}\matChatijk{l}{l}{\kappa,2l}\right)\Rxilbeta\Tr\Pxilortho\left(\Rxilbeta\wk-\Rxilalpha\wtwol\right),
\end{split}
\end{equation}
where the notation, for example $\matChatijk{l}{l}{\kappa,2l}$, implies that $\hat{\mat{C}}$ is constructed from the Roe average of the solutions $\qk$ and $\qtwol$.

The next theorem demonstrates that the discretization of the viscous terms leads 
to an entropy stable formulation.
\begin{theorem}\label{thrm:NS_Scon}
Assuming that the inviscid portion is entropy conservative, then the 
semi-discretization \eqref{eq:NSD} is entropy stable and freestream preserving. 
\end{theorem}
\begin{proof}
The proof of entropy stability  follows similarly to that given in Ref.~\cite{Fernandez2018prefine}. Here, we sketch the steps in the proof:
\begin{itemize}
\item Contract the viscous portions of Eq.~\eqref{eq:NSD} with $\wk\Tr\M$, contract Eq.~\eqref{eq:Theta} with 
$\left(\Thetaak{n}{\kappa}\right)\Tr\matChatijk{l}{n}{k}\M$ for $n=1,2,3$, and add the result.
\item After simplification, the result is of the form
\[
\frac{1}{2}\wk\Tr\frac{\mr{d}}{\mr{d}t}\left(\matJk\qk\right)+\wk\Tr\frac{1}{2}\matJk\frac{\mr{d}\qk}{\mr{d}t}=
-\sum\limits_{l,n=1}^{3}\left(\Thetaak{n}{\kappa}\right)\Tr\matChatijk{l}{n}{k}\M\Thetaak{n}{\kappa}+CT,
\]
where the first term on the right-hand side can be shown to be negative semi-definite and $CT$ represents the coupling terms from the SATs.
\item Summing over the elements, it can be shown that the interface coupling terms cancel and assuming 
appropriate boundary SATs, the remaining terms at the boundaries are consistent with the continuous analysis.
\end{itemize}

The proof of freestream preservation is identical to that of Theorem \ref{thrm:EulerFS},  taking into account the fact that 
all viscous fluxes and SAT terms vanish if the state vector is constant.
\end{proof}

%\begin{remark}
%The proposed semi-discrete schemes and the entropy stability proofs (see Theorems \ref{thrm:EulerScon} and \ref{thrm:NS_Scon}) can be %directly generalized to  fully discrete formulation if the time-derivative terms in the Navier-Stokes and GCL equations are discretized by using the %entropy conservative SBP operators developed in \cite{Friedrich2018}.
%\end{remark}

\section{Implementation}\label{sec:implementation}
%================================
\subsection{Approximation of the metrics}\label{sec:metrics}
%--------------------------------------------------------------------------
As follows from Theorem \ref{thrm:EulerScon}, the discrete GCL equations (\ref{eq:GCLtaudiss}--\ref{eq:GCLxildiss}) must be satisfied exactly.
Note that this property is not immediately satisfied. Several approaches have been proposed in the literature to address this problem 
\cite{Thomas1979, Visbal2002, Abe2014, Sjogreen2014}. These studies suggest that the following three properties should hold to guarantee that the GCL equations are satisfied exactly at the discrete level:
\begin{enumerate}
\item The discrete operators used to approximate the derivatives of the inviscid fluxes  should also be used to approximate the metric coefficients;
\item The metric tensor $\bm{\xi}_{\bm{x}}$ cannot be calculated directly by inverting  $\bm{x}_{\bm{\xi}}$, but instead it should be discretized so that 
it satisfies the GCL equation (\ref{eq:GCLxildiss});
\item The grid Jacobian  should be
computed by integrating the first discrete GCL equation (\ref{eq:GCLtaudiss}) using the same time integrator employed for approximating the governing equations.
\end{enumerate}
To satisfy the first two properties, the metrics in the present analysis are discretized by using the approach developed in \cite{Thomas1979}, which is selected because of its 
computational efficiency. Using the SBP spectral collocation operators defined in Section \ref{sec:notation}, the metric tensor  $\bm{x}_{\bm{\xi}}$ is 
discretized as follows:
\begin{equation}\label{eq:metrics}
\bm{x}_{\bm{\xi}} = \left[
\begin{array}{lll}
\mat{D}_{\xione}x_1 & \mat{D}_{\xitwo}x_1 & \mat{D}_{\xithree}x_1 \\
\mat{D}_{\xione}x_2 & \mat{D}_{\xitwo}x_2 & \mat{D}_{\xithree}x_2 \\
\mat{D}_{\xione}x_3 & \mat{D}_{\xitwo}x_3 & \mat{D}_{\xithree}x_3 
\end{array}
\right].
\end{equation}
The inverse metric coefficients are then calculated using the following formulas:
\begin{equation}\label{eq:Inv_metrics}
\begin{array}{l}
\mat{D}_{x_1}\xione =  J \left( \mat{D}_{\xithree} \left[\mat{D}_{\xitwo}x_2\right]  x_3 - \mat{D}_{\xitwo} \left[\mat{D}_{\xithree}x_2\right]  x_3 \right), \\
\mat{D}_{x_2}\xione =  J \left( \mat{D}_{\xitwo} \left[\mat{D}_{\xithree}x_1\right]  x_3 - \mat{D}_{\xithree} \left[\mat{D}_{\xitwo}x_1\right]  x_3 \right), \\
\mat{D}_{x_3}\xione =  J \left( \mat{D}_{\xithree} \left[\mat{D}_{\xitwo}x_1\right]  x_2 - \mat{D}_{\xitwo} \left[\mat{D}_{\xithree}x_1\right]  x_2 \right), \\

\mat{D}_{x_1}\xitwo =  J \left( \mat{D}_{\xione} \left[\mat{D}_{\xithree}x_2\right]  x_3 - \mat{D}_{\xithree} \left[\mat{D}_{\xione}x_2\right]  x_3 \right), \\
\mat{D}_{x_2}\xitwo =  J \left( \mat{D}_{\xithree} \left[\mat{D}_{\xione}x_1\right]  x_3 - \mat{D}_{\xione} \left[\mat{D}_{\xithree}x_1\right]  x_3 \right), \\
\mat{D}_{x_3}\xitwo =  J \left( \mat{D}_{\xione} \left[\mat{D}_{\xithree}x_1\right]  x_2 - \mat{D}_{\xithree} \left[\mat{D}_{\xione}x_1\right]  x_2 \right), \\

\mat{D}_{x_1}\xithree =  J \left( \mat{D}_{\xitwo} \left[\mat{D}_{\xione}x_2\right]  x_3 - \mat{D}_{\xione} \left[\mat{D}_{\xitwo}x_2\right]  x_3 \right) ,\\
\mat{D}_{x_2}\xithree =  J \left( \mat{D}_{\xione} \left[\mat{D}_{\xitwo}x_1\right]  x_3 - \mat{D}_{\xitwo} \left[\mat{D}_{\xione}x_1\right]  x_3 \right) ,\\
\mat{D}_{x_3}\xithree =  J \left( \mat{D}_{\xitwo} \left[\mat{D}_{\xione}x_1\right]  x_2 - \mat{D}_{\xione} \left[\mat{D}_{\xitwo}x_1\right]  x_2 \right) ,
\end{array}
\end{equation}
where the Jacobian $J$ is the determinant of the metric tensor $\bm{x}_{\bm{\xi}}$ given by Eq. (\ref{eq:metrics}).
The metric coefficients given by Eq. (\ref{eq:Inv_metrics}) satisfy the second discrete GCL equation (\ref{eq:GCLxildiss}) exactly.
Though there are other formulas for calculating the metric coefficients that identically satisfy Eq. (\ref{eq:GCLxildiss})  (e.g., see \cite{Abe2014, Sjogreen2014}),
these alternative approaches are more computationally expensive and have therefore not been considered in the present analysis.

As has been noted above, the grid Jacobian $J$ should satisfy the GCL equations (\ref{eq:GCLtaudiss}--\ref{eq:GCLxildiss}) exactly. 
There are at least three possible options for calculating $J$, which may provide this property. 

The first option is to discretely integrate Eq. (\ref{eq:GCLtaudiss}) by using the same temporal and spatial derivative
operators used for discretization of the governing equations. This numerically integrated Jacobian is then used to calculate
the inverse metric coefficients given  by  Eq. (\ref{eq:Inv_metrics}). Though this approach satisfies Eq. (\ref{eq:GCLxildiss})
exactly, Eq. (\ref{eq:GCLtaudiss}) may not be satisfied identically, because the integrated and determinant-based Jacobians
differ from each other by the magnitude of the order of the truncation error. 
%Note, however, that both discrete Jacobians approximate the exact Jacobian with the design order of accuracy. 
Our numerical results show that this integrated Jacobian approach provides
the freestream preservation only with the accuracy that is 2--3 orders of magnitude larger than the machine zero. This is the main reason why this 
approach is not used in the present analysis.

The second approach uses the  integrated and determinant-based Jacobians simultaneously. The key difference of this approach from the previous one is that both the spatial
terms in the discretized Navier-Stokes equations (\ref{eq:skewEulersemi})  and the mesh velocity terms in Eq. (\ref{eq:GCLtaudiss}) 
are calculated by using the determinant-based Jacobian, while the Jacobian in the time derivative term is computed by integrating
Eq. (\ref{eq:GCLtaudiss}). This method uses the conservative discretization in time, satisfies both GCL equations exactly, and provides the freestream preservation to machine accuracy. This dual-Jacobian approach is used in all numerical experiments presented in Section \ref{sec:results}.

The last option considered is based on a nonconservative approximation of the time derivative term in Eq. (\ref{eq:Inv_metrics}). Using the
product rule, the time derivative term can be represented as follows:
\begin{equation*}
\frac{\mr{d}(\matJk\qk)}{\mr{d}\tau} = \matJk\frac{\mr{d}\qk}{\mr{d}\tau}  + \left[\frac{\mr{d}\matJk}{\mr{d}\tau}\right]\qk
\end{equation*}
In this approach, the grid Jacobian  is evaluated by calculating the determinant of the metric tensor $\bm{x}_{\bm{\xi}}$. The 
time derivative of the Jacobian, $\frac{\mr{d}\matJk}{\mr{d}\tau}$, is computed using the GCL equation (\ref{eq:GCLtaudiss}). 
Though this approach satisfies both GCL equations exactly, it is nonconservative in time and therefore not used in the present analysis. 

\subsection{Entropy conservative fluxes}\label{sec:Efluxes}
%--------------------------------------------------------------------------
As follows from Theorem \ref{thrm:AlmkDxil}, the two-point entropy conservative flux functions $\bfnc{U}^{sc}$ and $\bfnc{F}^{sc}_{\xm}$ should be  symmetric, consistent, and 
satisfy the shuffle conditions \eqref{eq:Uflux} and \eqref{eq:Fflux}, respectively. For the Euler and Navier-Stokes equations, a computationally  affordable entropy consistent 
flux function $\bfnc{F}^{sc}_{\xm}$ satisfying all the required constraints has been derived by Ismail and Roe in \cite{Ismail2009}.
It should be noted that this two-point  entropy consistent flux function $\bfnc{F}^{sc}_{\xm}$ is not unique.
We refer the reader to  \cite{Ranocha2018} for a review of alternative entropy consistent flux functions satisfying \eqref{eq:Fflux}. 

Using the same technique developed in \cite{Ismail2009} for constructing  $\bfnc{F}^{sc}_{\xm}$, the following two-point entropy consistent fluxes  satisfying 
the shuffle condition \eqref{eq:Uflux} can be derived:

\begin{equation}\label{eq:Usc_flux1}
\bfnc{U}_1^{sc} = \left[
\begin{array}{c}
\overline{Z_5}_{\rm{log}}\overline{Z_1} \\
\overline{Z_5}_{\rm{log}}\overline{Z_2} \\
\overline{Z_5}_{\rm{log}}\overline{Z_3} \\
\overline{Z_5}_{\rm{log}}\overline{Z_4} \\
\frac 12 \overline{Z_5}_{\rm{log}}\left( \frac{R(\gamma+1)}{\overline{Z_1}_{\rm{log}}(\gamma-1)} + \frac 1{\overline{Z_1}}\left( \overline{Z_2}^2+ \overline{Z_3}^2 +\overline{Z_4}^2
- R\frac{\overline{Z_5}}{ \overline{Z_5}_{\rm{log}}}  \right) \right) 
\end{array}
\right] ,
\end{equation}

\begin{equation}\label{eq:Usc_flux2}
\bfnc{U}_2^{sc} = \left[
\begin{array}{c}
\overline{Z_1}_{\rm{log}} \\
\overline{Z_1}_{\rm{log}} \overline{Z_2} \\
\overline{Z_1}_{\rm{log}} \overline{Z_3} \\
\overline{Z_1}_{\rm{log}} \overline{Z_4} \\
\overline{Z_1}_{\rm{log}}\left( \frac{R}{(\gamma-1) \overline{ Z_5}_{\rm{log}}} + \overline{Z_2}^2+ \overline{Z_3}^2 +\overline{Z_4}^2 -\frac 12 \left( \overline{Z_2^2}+ \overline{Z_3^2} +\overline{Z_4^2} \right) \right) 
\end{array}
\right],
\end{equation}
where 
$\fnc{Z}=\left[\frac{1}{\sqrt{T}}, \frac{V_1}{\sqrt{T}},\frac{V_2}{\sqrt{T}},\frac{V_3}{\sqrt{T}}, \rho\sqrt{T}  \right]\Tr$ for  $\bfnc{U}_1^{sc}$ and
$\fnc{Z}=\left[\rho, V_1, V_2, V_3, \frac 1T \right]\Tr$ for  $\bfnc{U}_2^{sc}$.
In the above equations, $\overline{Z}$ and $\overline{Z}_{\rm{log}}$ are the arithmetic and logarithmic averages:
\begin{equation}\label{eq:Avr}
\begin{array}{l}
 \overline{Z}= \frac{Z_L+Z_R}2 \\
\overline{Z}_{\rm{log}}= \frac{Z_R - Z_L}{\log(Z_R) - \log(Z_L)} ,
\end{array}
\end{equation}
where the subscripts $L$ and $R$ correspond to two arbitrary, yet physically meaningful, states of the flow variables.

\begin{lemma}\label{thrm:Uflux}
The two-point fluxes  given by Eqs. (\ref{eq:Usc_flux1}) and (\ref{eq:Usc_flux2}) are symmetric, consistent, and satisfy the shuffle condition \eqref{eq:Uflux}.
\end{lemma}
\begin{proof}
The symmetry of the entropy consistent fluxes  $\bfnc{U_1}^{sc}$ and $\bfnc{U_2}^{sc}$ follows immediately form the symmetry of the arithmetic
and logarithmic averages given by Eq. \eqref{eq:Avr}. The consistency can be directly verified by substituting $Z_L=Z_R=Z$ in Eqs. (\ref{eq:Usc_flux1}--\ref{eq:Usc_flux1})
and using the L'Hopital's rule for evaluating the logarithmic averages.
The shuffle condition \eqref{eq:Uflux} can be verified directly by substituting  Eqs. (\ref{eq:Usc_flux1}--\ref{eq:Usc_flux1}) in Eq. \eqref{eq:Uflux}.
\end{proof}
\begin{remark}
 As evident from Eqs. (\ref{eq:Usc_flux1}--\ref{eq:Usc_flux1}), the two-point entropy consistent flux  $\bfnc{U}^{sc}$ is not unique and
other two-point  flux functions satisfying Eq. \eqref{eq:Uflux} can be readily constructed by using different auxiliary flow variables $Z$.
\end{remark}

\section{Results}\label{sec:results}
%=====================
The accuracy and freestream preservation properties of the proposed high-order entropy stable spectral collocation schemes 
for the 3-D Euler and Navier-Stokes equations on moving and deforming grids are tested using 
standard benchmark problems including isentropic vortex and viscous shock flows.  
In all simulations presented herein, the low-storage, five-stage, fourth-order Runge-Kutta
scheme developed in \cite{KCL2000} is used to advance the semi-discretization in time. Note
that this scheme violates the entropy stability property of the semi-discrete
operator by a factor proportional to the local temporal truncation error. In the present analysis, the temporal error 
is evaluated and monitored by using the Runge-Kutta built-in error estimator. For all test problems considered, 
the temporal error is at least 2-3 orders of magnitude less than the corresponding spatial error.
The spatial derivatives in the Euler and Navier-Stokes equations  are discretized using the entropy stable 
spectral collocation schemes with  two-point entropy conservative flux given by Eq. \eqref{eq:Usc_flux2} and
the polynomial bases of various degrees ($p = 3, 4, 5$).
It should be emphasized that no additional artificial dissipation other than that introduced by the entropy 
consistent Roe-type flux (see  Eq. \eqref{eq:diss}) at element interfaces is used for all numerical experiments presented in this paper.

\subsection{Inviscid flows}\label{sec:inviscid}
%------------------------------------------------------
To evaluate the accuracy of the new entropy stable spectral collocations schemes for the Euler equations,
we consider the propagation of an isentropic vortex in a moving deforming domain. The isentropic vortex is an 
exact solution of the Euler equations, which is given by 
\begin{figure}[htbp]
\label{fig:vortex_contours}
  \hspace{-2ex}
  \includegraphics[width=6.5cm]{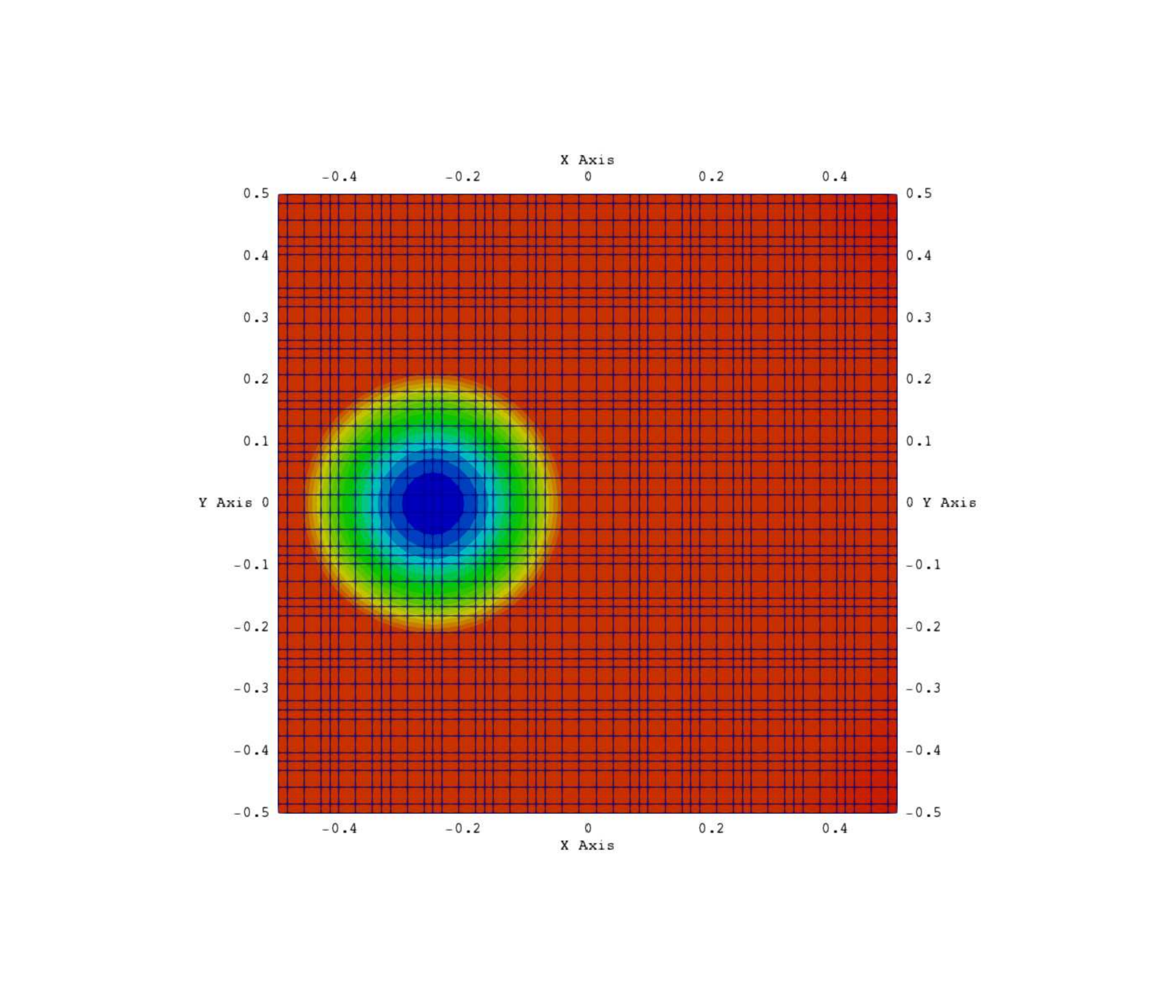}
  \includegraphics[width=6.5cm]{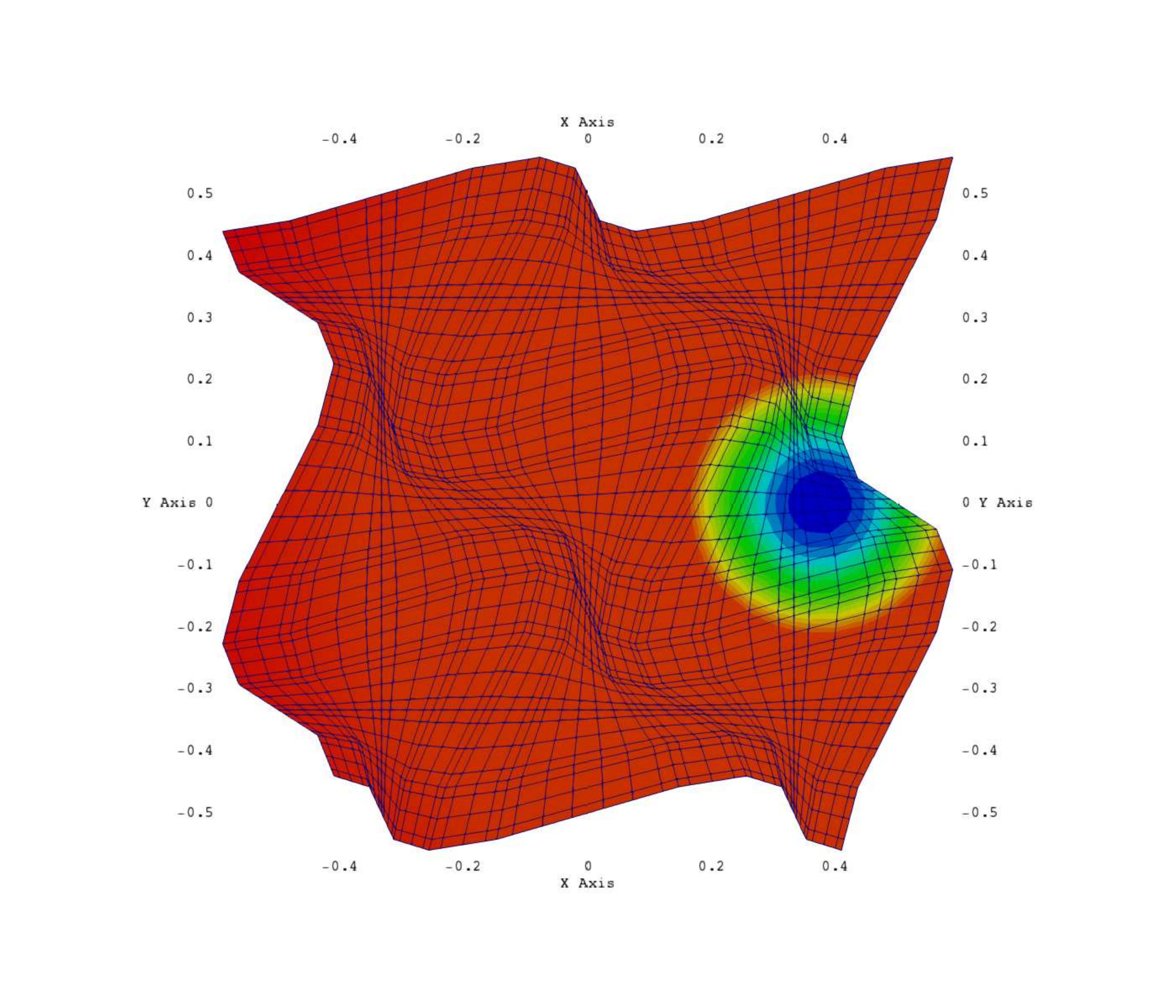}
  \caption{Density contours obtained with the entropy stable spectral collocation 
  scheme ($p=4$) on a dynamic grid at $t=0$ (left panel) and $t=2.5$ for the isentropic vortex problem.}
\end{figure}
\begin{equation}\label{eq:vortex}
\left[
\begin{array}{l}
\rho \\ V_1 \\ V_2\\ V_3\\ T
\end{array}
\right] =
\left[
\begin{array}{c}
T^{\frac 1{\gamma-1}}\\ 
V_{\infty} \cos{\alpha} - \epsilon \frac{y-y_0 -V_{\infty}t\sin{\alpha}}{2 \pi}\exp{\left(\frac{g(x,y,z,t)}2\right)} \\ 
V_{\infty} \sin{\alpha} - \epsilon \frac{x-x_0 -V_{\infty}t\cos{\alpha}}{2 \pi}\exp{\left(\frac{g(x,y,z,t)}2\right)} \\ 
0\\ 
1-\epsilon^2 M^2_{\infty}\frac{\gamma-1}{8\pi^2}\exp{\left(g(x,y,z,t)\right)}
\end{array}
\right],  \\
\end{equation}
\begin{equation*}
g(x,y,z,t) = 1 - \left( x-x_0-V_{\infty}t \cos{\alpha} \right)^2 - \left( y-y_0-V_{\infty}t \sin{\alpha} \right)^2,
\end{equation*}
For this test problem, we set $M_{\infty} = 0.5, V_\infty=0.25, \epsilon = 5.0, x_0 = -0.25, y_0 = 0, \alpha=0$.

The dynamic grid is constructed by mapping vertices of each 
grid cell of the uniform Cartesian grid in the computational domain $\hat{\Omega} = [-1, 1] \times [-1, 1] \times [-1, 1]$
onto the corresponding vertices of a moving deforming hexahedral grid element by using
the following coordinate transformation:
\begin{equation}\label{eq:mapping_vortex}
\begin{split}
& x(\xi, \eta, \zeta, \tau) =  \frac 12 \xi   + A_1 \sin(\omega \tau) \cos\left(\nu\xi - \frac{\pi}4\right)\cos\left(\nu\eta-\frac{\pi}4\right)   \\
& y(\xi, \eta, \zeta, \tau) =  \frac 12 \eta + A_2 \sin(\omega \tau) \cos\left(\nu\xi - \frac{\pi}4\right)\cos\left(\nu\eta-\frac{\pi}4\right)  \\
& z(\xi, \eta, \zeta, \tau) = \frac{\zeta}K,
\end{split}
\end{equation}
where $K$ is the number of elements in the $\xi-$ and $\eta-$directions
and the parameters $A_1$, $A_2$, $\nu$, and $\omega$ are set equal to $0.09$, $0.06$, $3\pi/4$, and $\pi$, respectively.
Note that only the vertices of each hexahedral grid element  are moved  according to Eq. \eqref{eq:mapping_vortex}, 
while the element edges are constructed by connecting the corresponding vertices with straight lines 
and the interior grid points are generated by using the 3-D transfinite interpolation.
As a result, the grid is nonsmooth across all element interfaces.

Though this test problem is two-dimensional in nature, a 3-D  code is used to solve this isentropic vortex flow
for testing the proposed entropy stable spectral collocation schemes on dynamic unstructured grids. In this study,
the Dirichlet boundary conditions are imposed at the moving boundaries along the $\xi-$ and $\eta-$directions
by using the exact solution given by Eq. \eqref{eq:vortex} and the periodic boundary conditions are
used in the $\zeta$-direction. 
The mesh velocity vector appearing in both the proposed numerical scheme \eqref{eq:skewEulersemi} and the GCL equation \eqref{eq:GCLtime}
 is computed analytically by differentiating the time-dependent coordinate transformation given by Eq. \eqref{eq:mapping_vortex}.
\begin{table}[h]
\label{tbl:vortex}
\begin{center}
\begin{tabular}{ccccc}
\hline
 $p = 3$              & $L_2$ error                     & $L_2$ rate    & $L_{\infty}$ error            & $L_{\infty}$ rate   \\
\hline
 $6   \times  6$    & $3.92\times 10^{-5}$     & --                 & $4.26\times 10^{-4}$     &  --                         \\
 $12 \times 12$   & $6.50\times 10^{-6}$     & 2.59             & $9.99\times 10^{-5}$      & 2.09                     \\
 $24 \times 24$   & $7.85\times 10^{-7}$     & 3.05             & $1.55\times 10^{-5}$      & 2.69                     \\
 $48 \times 48$   & $7.85\times 10^{-8}$     & 3.32             & $1.58\times 10^{-6}$      & 3.29                     \\
\hline
\\
 $p = 4$              &                                        &                    &                                       &                             \\
\hline
 $6   \times  6$    & $1.71\times 10^{-6}$      & --                 & $2.63\times 10^{-5}$     &  --                         \\
 $12 \times 12$   & $1.88\times 10^{-7}$      & 3.18             & $3.27\times 10^{-6}$      & 3.00                     \\
 $24 \times 24$   & $1.26\times 10^{-8}$      & 3.89             & $3.39\times 10^{-7}$      & 3.27                     \\
 $48 \times 48$   & $7.00\times 10^{-10}$    & 4.17             & $2.23\times 10^{-8}$      & 3.92                     \\
\hline
\\
 $p = 5$              &                                        &                    &                                       &                             \\
\hline
 $6   \times  6$    & $6.69\times 10^{-8}$      & --                 & $1.00\times 10^{-6}$     &  --                        \\
 $12 \times 12$   & $3.86\times 10^{-9}$      & 4.12             & $9.74\times 10^{-8}$     & 3.36                     \\
 $24 \times 24$   & $1.26\times 10^{-10}$    & 4.93             & $3.34\times 10^{-9}$     & 4.82                     \\
 $48 \times 48$   & $3.19\times 10^{-12}$    & 5.31             & $7.73\times 10^{-11}$    & 5.48                     \\
\hline
\end{tabular}
\end{center}
\caption{Error convergence on dynamic grids for the isentropic vortex problem.}
\end{table}

At the initial moment of time $t=0$, the vortex is centered at $x_0=-0.25, y_0=0$
and the grid is not deformed. The final time is set equal to $2.5$, so that the grid undergoes $1\frac 14$ deformation
cycles and reaches its maximum deformation, while the vortex propagates throughout the entire domain. Density contours
obtained using the new entropy stable spectral collocation scheme for $p=4$
on the $12\times 12\times 1$ dynamic grid at $t = 0$ and $t=2.5$ are shown in Fig. \ref{fig:vortex_contours}. As one can see in the figure,
the grid elements are highly deformed and skewed at $t=2.5$. Despite this large unsteady grid deformations, the numerical solution
is free of spurious oscillations and preserves the symmetry of the vortex.

The error convergence rate of the new high-order entropy stable spectral collocation schemes
is evaluated on a sequence of moving deforming grids with $6\times 6\times1$,  $12\times 12\times1$, 
$24\times 24\times1$, $48\times 48\times1$, elements. The refined dynamic grids are obtained by constructing
a sequence of nested grids in the computational domain and then mapping them onto the physical domain by 
using Eq. \eqref{eq:mapping_vortex}. As mentioned above, all 3-D grids used 
in this grid refinement study are one--element thick in the $\zeta$ direction and their thickness is proportionally 
reduced as the grid is globally refined. 
\begin{figure}[htbp]
\label{fig:Inviscid_freestream}
  \centering
  \includegraphics[width=8.cm]{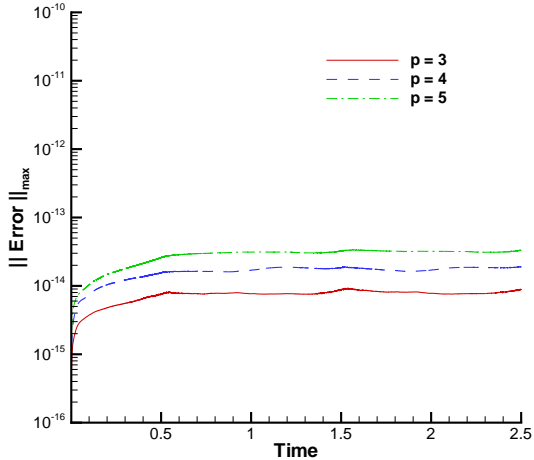}
  \caption{Freestream preservation provided by the entropy stable spectral collocation 
  schemes for the Euler equations on the $12 \times 12 \times 1$ dynamic grid.}
\end{figure}

The $L_2$ and $L_{\infty}$ errors  of the entire solution vector obtained with the new entropy stable schemes for polynomial bases of degree $p = 3, 4, 5$ on globally
refined dynamic grids are presented in Table \ref{tbl:vortex}. The convergence rate in the range of $p$ to $p+0.5$ 
is achieved on  highly deformed, skewed moving grids for all polynomial degrees considered. It should be noted that the $(p+1)$th-order convergence or so-called
superconvergence is not attainable on highly deformed skewed moving meshes.

We now corroborate our theoretical results presented in Theorem \ref{thrm:EulerFS} and show that the proposed spectral collocation schemes
for the Euler equations preserve the freestream solution on the same dynamic grids used for the isentropic vortex problem. Figure \ref{fig:Inviscid_freestream}
shows the $L_{\infty}$ error of the entire vector of conservative variables computed using the new entropy stable schemes \eqref{eq:skewEulersemi} for $p = 3, 4, 5$
on the $12\times 12 \times 1$ grid  given by Eq. \eqref{eq:mapping_vortex} for the constant inviscid flow at $M_{\infty}=0.5$.
As evident from this comparison, the constant solution is preserved on the moving deforming grid with the accuracy of $O(10^{-14})$ for all polynomial degrees considered
over the entire interval of integration. As one can see in Fig.~\ref{fig:Inviscid_freestream}, there is a slight increase in error as the polynomial degree increases. This increase 
can be explained by accumulation of the roundoff error, since the higher order schemes require larger number of operations per
degree of freedom.

\subsection{Viscous flows}\label{sec:viscous}
%------------------------------------------------------
We now evaluate the accuracy and freestream preservation properties of the proposed entropy stable spectral collocation schemes for
the Navier-Stokes equations on moving deforming unstructured hexahedral grids. 
To assess the effect of grid motion and deformation on the error convergence of the high-order spectral collocation schemes,
the propagation of a viscous shock is considered. For Prandtl number equal to $3/4$, the 1-D compressible Navier-Stokes equations reduce to the single 
momentum equation, which can be solved exactly and used to measure the accuracy of the numerical scheme under consideration 
(e.g., see \cite{Fisher2012phd}). The 3-D exact solution is then obtained by rotating the 1-D viscous shock profile, which is initially located in the middle of the domain,  by $45^\circ$ with respect to both $x-$ and $z-$axes. 
 For this test problem, the Dirichlet boundary conditions are used at all physical boundaries and
 the Reynolds and Mach numbers are set equal to $10$ and $2.5$, respectively.
\begin{figure}[htbp]
\label{fig:vshock_contours}
  \hspace{-2ex}
  \includegraphics[width=6.5cm]{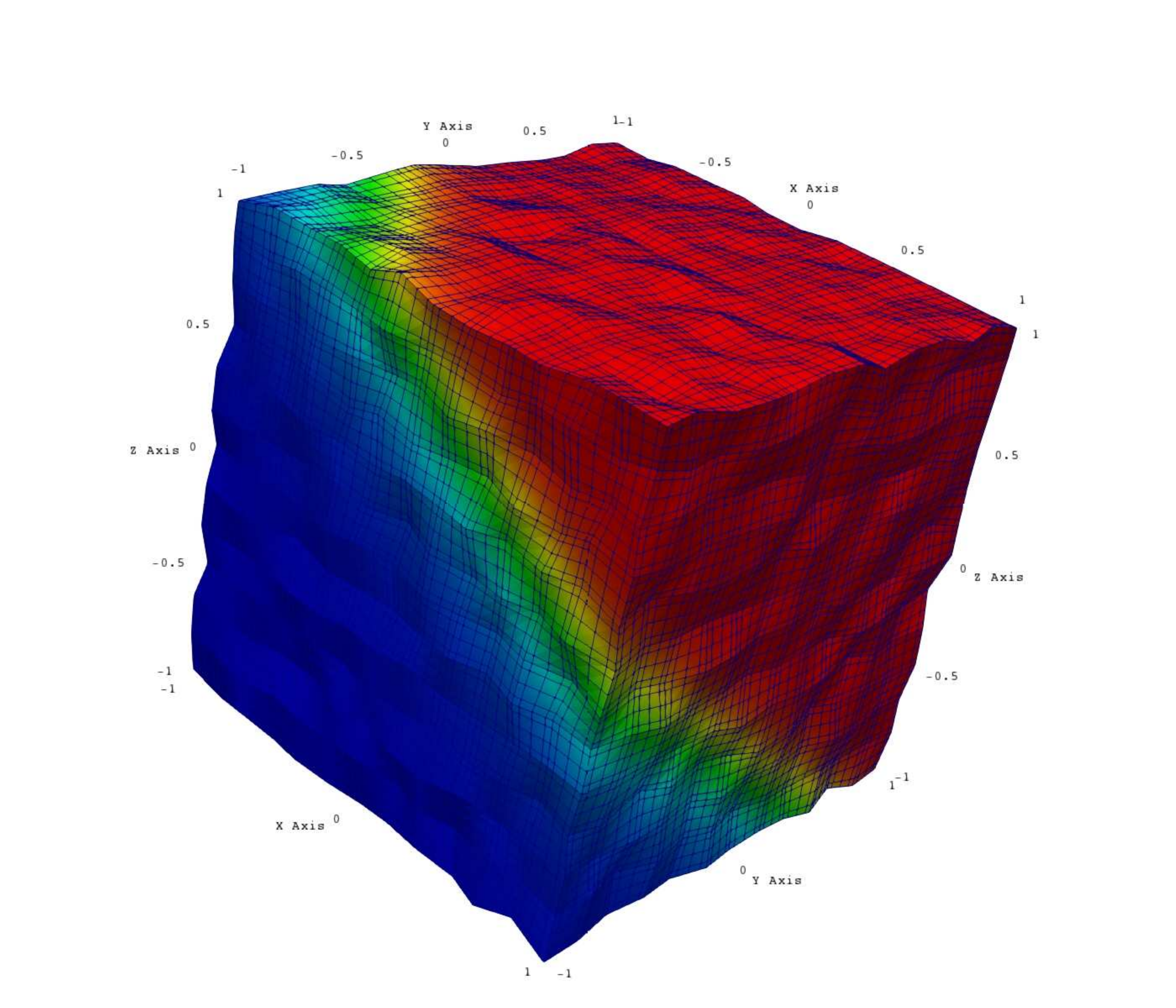}
 \includegraphics[width=6.5cm]{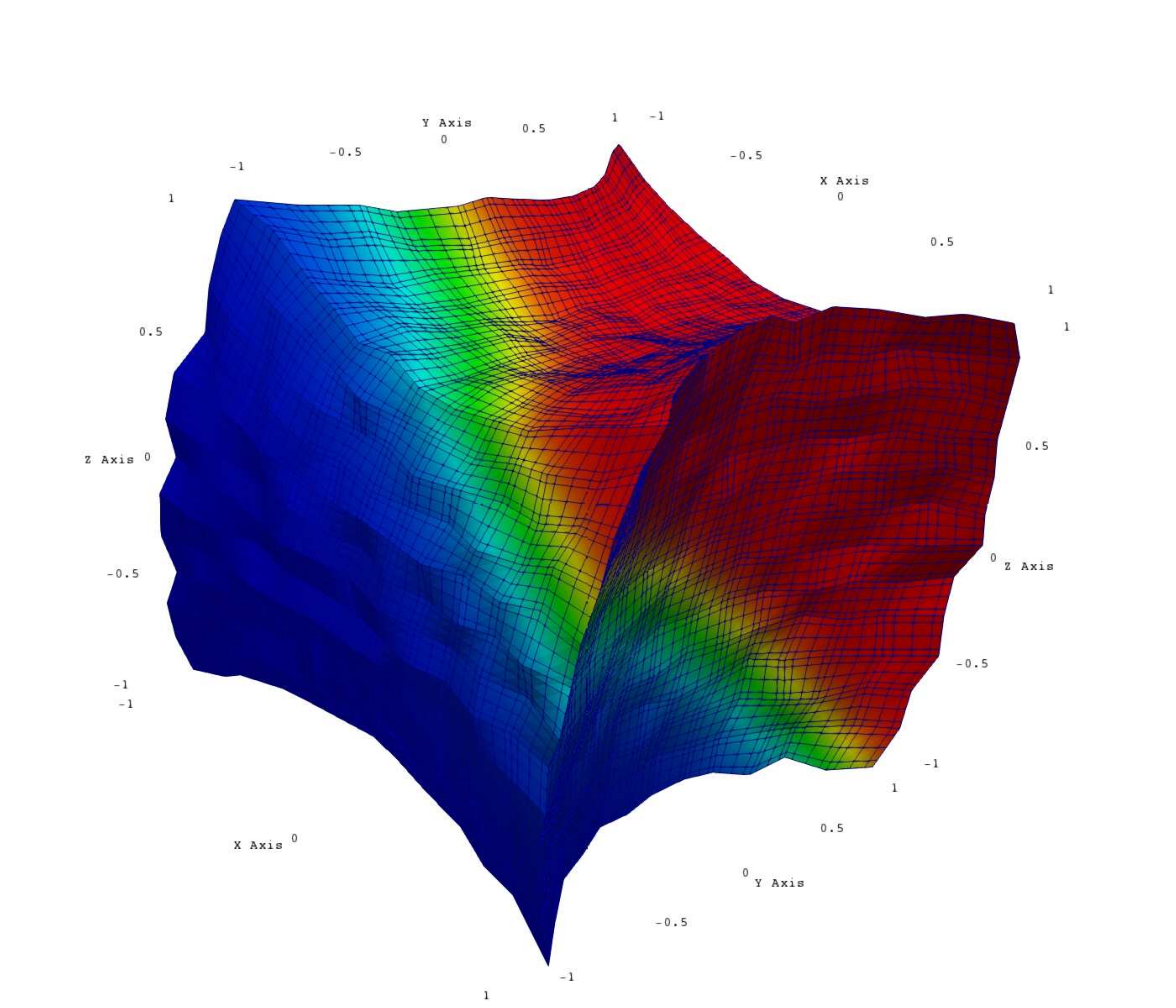}
  \caption{Density contours obtained with the entropy stable spectral collocation 
  scheme ($p=4$) on a dynamic grid at $t=0$ (left panel) and $t=2.5$ for the viscous shock problem.}
\end{figure}
\begin{table}[h]
\label{tbl:vshock}
\begin{center}
\begin{tabular}{ccccc}
\hline
 $p = 3$              & $L_2$ error                     & $L_2$ rate    & $L_{\infty}$ error            & $L_{\infty}$ rate   \\
\hline
 $6   \times  6 \times 6$      & $1.99\times 10^{-3}$     & --                 & $8.33\times 10^{-2}$     &  --                         \\
 $12 \times 12 \times 12$   & $1.95\times 10^{-4}$     & 3.36             & $1.37\times 10^{-2}$      & 2.60                     \\
 $24 \times 24 \times 24$   & $1.85\times 10^{-5}$     & 3.39             & $1.22\times 10^{-3}$      & 3.49                     \\
 $48 \times 48 \times 48$   & $2.56\times 10^{-6}$     & 2.86             & $2.12\times 10^{-4}$      & 2.52                     \\
\hline
\\
 $p = 4$              &                                        &                    &                                       &                             \\
\hline
 $6   \times  6 \times 6  $    & $4.05\times 10^{-4}$      & --                 & $1.52\times 10^{-2}$     &  --                         \\
 $12 \times 12 \times 12$   & $2.55\times 10^{-5}$      & 3.99             & $2.09\times 10^{-3}$      & 2.86                     \\
 $24 \times 24 \times 24$   & $1.19\times 10^{-6}$      & 4.43             & $1.35\times 10^{-4}$      & 3.95                     \\
 $48 \times 48 \times 48$   & $7.04\times 10^{-8}$      & 4.07             & $9.71\times 10^{-6}$      & 3.80                     \\
\hline
\\
 $p = 5$              &                                        &                    &                                       &                             \\
\hline
 $6   \times  6 \times 6  $    & $9.62\times 10^{-5}$     & --                 & $5.36\times 10^{-3}$     &  --                        \\
 $12 \times 12 \times 12$   & $3.53\times 10^{-6}$     & 4.77             & $6.53\times 10^{-4}$     & 3.04                    \\
 $24 \times 24 \times 24$   & $7.42\times 10^{-8}$     & 5.57             & $1.24\times 10^{-5}$     & 5.72                     \\
 $48 \times 48 \times 48$   & $2.28\times 10^{-9}$     & 5.02             & $6.19\times 10^{-7}$     & 4.32                     \\
\hline
\end{tabular}
\end{center}
\caption{Error convergence on dynamic randomly perturbed grids for the viscous shock problem.}
\end{table}

Moving grids for the viscous shock problem are generated by perturbing the 
vertices of a uniform Cartesian grid by $25\%$ of the grid spacing  in a randomly chosen direction and mapping them onto
vertices of a dynamic grid in the physical domain by using the following coordinate transformation:
\begin{equation}\label{eq:mapping_vshock}
\begin{split}
& x(\xi, \eta, \zeta, \tau) =  \xi      + A_1 \sin(\omega \tau) \xi (\eta-1)(\eta+1)   \\
& y(\xi, \eta, \zeta, \tau) =  \eta   + A_2 \sin(\omega \tau) \eta (\zeta-1)(\zeta+1)  \\
& z(\xi, \eta, \zeta, \tau) =  \zeta + A_3 \sin(\omega \tau) \zeta (\xi-1)(\xi+1),  \\
\end{split}
\end{equation}
where the parameters $A_1$, $A_2$, $A_3$, and $\omega$ are set equal to $0.4$, $-0.2$, $0.3$, and $\pi$, respectively.
Similar to the previous test problem, all moving grids for the viscous shock problem are straight-sided and 
the mesh velocities are calculated by differentiating Eq. \eqref{eq:mapping_vshock}.
\begin{figure}[htbp]
\label{fig:Viscous_freestream}
  \centering
  \includegraphics[width=8.cm]{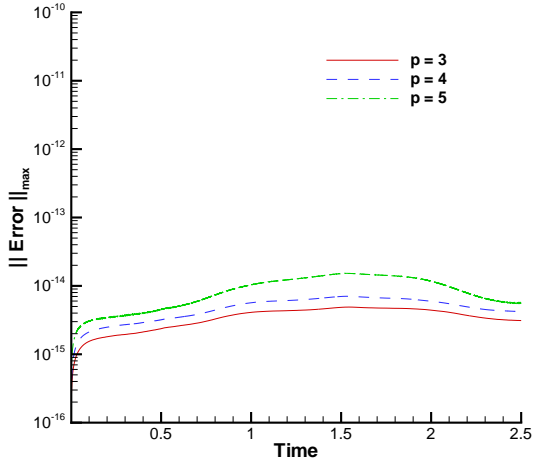}
  \caption{Freestream preservation provided by the entropy stable spectral collocation 
  schemes for the Navier-Stokes equations on the $12 \times 12 \times 1$ dynamic grid.}
\end{figure}

The 3-D unsteady Navier-Stokes equations are integrated until $t=0.5$, which corresponds to an instant in time when the grid reaches the maximum deformation.
Figure \ref{fig:vshock_contours} shows density contours computed using the proposed entropy stable spectral collocation scheme 
for $p=4$ on the randomly perturbed $12\times 12\times 12$ grid at $t=0$ and $t=0.5$. 
The numerical solution obtained on this moving and deforming grid is free of spurious oscillations and accurately approximates 
the shock propagation in the domain.

To evaluate the accuracy of the new entropy stable schemes for the Navier-Stokes equations, we measure the $L_2$ and $L_{\infty}$ errors 
for the viscous shock problem on a sequence of the globally refined randomly perturbed moving grids. Since each grid is individually perturbed, 
the grids are not nested. As follows from our numerical results presented in Table~\ref{tbl:vshock}, the proposed entropy stable schemes
demonstrate in average $p$th-order convergence for all polynomial degrees considered. It should be noted that the error
convergence is nonmonotonic because of the random perturbations  introduced in dynamic grids used in this grid refinement study.

Similar to the Euler equations, a constant flow is an exact solution of the unsteady Navier-Stokes equations. 
To demonstrate that the new entropy stable spectral collocation schemes preserve a freestream viscous flow,
we integrate the unsteady Navier-Stokes equations on the same randomly perturbed dynamic grids generated for the 
viscous shock problem. The $L_{\infty}$ errors of the entire solution vector computed with the new spectral collocation
scheme given by Eq. \eqref{eq:NSD} for $p = 3, 4, 5$ on the $12 \times 12 \times 12$ randomly perturbed moving grid 
are presented in Fig. \ref{fig:Viscous_freestream}. Our numerical results show that the proposed entropy stable scheme
preserves the constant solution with the machine accuracy for all polynomial degrees considered.
Similar to the inviscid flow test case,  there is a slight growth in the roundoff error accumulation as the polynomial degree increases.

\section{Conclusions}\label{sec:conclusions}
%===========================
New provably stable spectral collocation schemes of arbitrary order of accuracy are developed for the unsteady 3-D Euler and Navier-Stokes equations 
on dynamic unstructured grids. Using the arbitrary Lagrangian-Eulerian formulation, the governing equations are mapped  onto a fixed reference system of coordinates.
To provide the design order of accuracy, entropy stability, and conservation on moving and deforming grids, the skew-symmetric form
of the Euler and Navier-Stokes equations are discretized by using high-order summation-by-parts spectral collocation operators. 
In addition to that, the entropy stability proof requires that the geometric conservation laws (GCL) must be satisfied exactly at the discrete level and
the flux function has to satisfy an additional shuffle condition.
To satisfy the GCL equations exactly,  the metric coefficients are computed by using formulas that guarantee conservation, which are discretized 
by  the same SBP operators employed  to approximate the flux derivatives. 
Furthermore, the spatial derivative terms in the Euler and Navier-Stokes equations
are calculated using the determinant-based Jacobian, while the time derivative term is calculated using the Jacobian obtained by integrating the
corresponding GCL equation. This dual-Jacobian formulation satisfies all GCL equations simultaneously.
Two new two-point entropy conservative flux functions satisfying the required shuffle condition are derived for 
the 3-D Euler and Navier-Stokes equations. 
Our numerical results show that the developed entropy stable spectral collocation schemes are fully conservative, design-order accurate, 
and provide freestream preservation for both the Euler and Navier-Stokes equations on highly deforming, moving unstructured grids.

\appendix
\section{Entropy conservation of the semi-discrete form of the Euler equations}\label{app:Eulerentropycon}
%==================================================================
%\setlength{\belowdisplayskip}{0pt} \setlength{\belowdisplayshortskip}{0pt}
%\setlength{\abovedisplayskip}{0pt} \setlength{\abovedisplayshortskip}{0pt}
In this appendix, it is proven that the semi-discrete from~\eqref{eq:skewEulersemi} is entropy conservative 
for periodic domains.  If the scheme is entropy conservative on a periodic domain, then it telescopes the entropy 
to the boundaries on non-periodic domains and the analysis of entropy conservation/dissipation reduces to the analysis 
of the weak imposition of boundary conditions. 

The analysis follows in a one-to-one fashion the continuous analysis. The first step is to multiply by the entropy 
variables and discretely integrate in space; this is accomplished by multiplying~\eqref{eq:skewEulersemi} by $\wk\Tr\M$ 
which gives
\begin{multline}\label{eq:skewEulersemi1}
\overbrace{\wk\Tr\M\frac{\mr{d}}{\mr{d}\tau}\matJk\qk}^{I}+\wk\Tr\sum\limits_{l=1}^{3}\overbrace{\left(\Qxil\matBlk+\matBlk\Qxil\right)\circ\matQsc{\qk}{\qk}\one}^{II}\\
+\wk\Tr\sum\limits_{l,m=1}^{3}\overbrace{\left(\Qxil\Almk+\Almk\Qxil\right)\circ\matFxmsc{\qk}{\qk}\one}^{III}=\\
\wk\Tr\M\left(\bm{SAT}_{\tau}+\bm{SAT}_{\xil}\right).
\end{multline}
First, the term $I$ in~\eqref{eq:skewEulersemi1} reduces to
\begin{equation}\label{eq:termI}
\begin{split}
I =\onesca\Tr\Msca\frac{\mr{d}\matJksca\sk}{\mr{d}\tau}+\phik\Tr\Msca\frac{\mr{d}\vecJk}{\mr{d}\tau}.
\end{split}
\end{equation}
%where $\wk\Tr\qk$ is to be understood as the vector constructed as 
%\[
%\left(\wk\Tr\wk\right)(i)\equiv\wkmatrix{i}{:}\Tr\qkmatrix{i}{:},
%\]
where $\vecJk$ is a vector containing the metric Jacobian at the nodes.

Using the symmetry of $\matFxmsc{\qk}{\qk}$, 
the SBP property $\mat{Q}+\mat{Q}=\mat{E}$, 
Theorem~\ref{thrm:telescope}, the consistency of the 
derivative operator, $\mat{D}\bm{1}=\bm{0}\rightarrow\bm{1}\Tr\mat{Q}=\bm{1}\Tr\mat{E}$, and the relation, 
$\varphi=\bfnc{W}\Tr\bfnc{U}-\mathcal{S}$, the term $II$ reduces to 
\begin{equation}\label{eq:termII}
\begin{split}
&II = \phik\Tr\Msca\Dxilsca\matBlksca\onesca+\onesca\Tr\Msca\Dxilsca\matBlksca\sk-\onesca\Tr\Exilsca\matBlksca\wk\Tr\qk\\
&+\one\Tr\left(\Exil\matBlk\right)\circ\matQsc{\qk}{\qk}\wk.
\end{split}
\end{equation}
Using similar relations as for term $II$, term $III$ reduces to
\begin{equation}\label{eq:termIII}
\begin{split}
&III=\left(\psixmk\right)\Tr\Msca\Dxilsca\matAlmksca\onesca-\onesca\Tr\Exilsca\matAlmksca\psixmk\\
&+\one\Tr\left(\Exil\matAlmk\right)\circ\matFxmsc{\qk}{\qk}\wk.
\end{split}
\end{equation}
Inserting~\eqref{eq:termI}, \eqref{eq:termII}, \eqref{eq:termIII} into~\eqref{eq:skewEulersemi1} and rearranging gives
\begin{multline}\label{eq:skewEulersemi2}
\onesca\Tr\Msca\left[\frac{\mr{d}\matJksca\sk}{\mr{d}\tau}+\sum\limits_{l=1}^{3}\Dxilsca\matBlksca\sk\right]+
\phik\Tr\Msca\left[\frac{\mr{d}\vecJk}{\mr{d}\tau}+\sum\limits_{l=1}^{3}\Dxilsca\matBlksca\onesca\right]\\
+\left(\psixmk\right)\Tr\Msca\left[\sum\limits_{l,m=1}^{3}\Dxilsca\matAlmksca\onesca\right]
=\\
\sum\limits_{l=1}^{3}\left[\onesca\Tr\Exilsca\matBlksca\wk\Tr\qk-\one\Tr\left(\Exil\matBlk\right)\circ\matQsc{\qk}{\qk}\wk\right]\\
+\sum\limits_{l,m=1}^{3}\left[\onesca\Tr\Exilsca\matAlmksca\psixmk-\one\Tr\left(\Exil\matAlmk\right)\circ\matFxmsc{\qk}{\qk}\wk\right]\\
+\wk\Tr\M\left(\bm{SAT}_{\tau}+\bm{SAT}_{\xil}\right).
\end{multline}
Ignoring for now the terms on the right-hand side of~\eqref{eq:skewEulersemi2}, since the term 
\begin{equation*}
\begin{split}
&\onesca\Tr\Msca\left[\frac{\mr{d}\matJksca\sk}{\mr{d}\tau}+\sum\limits_{l=1}^{3}\Dxilsca\matBlksca\sk\right]\approx\\
&\int_{\Ohatk}\frac{\partial\Jk\mathcal{S}}{\partial\tau}+\sum\limits_{l=1}^{3}\frac{\partial}{\partial\xil}\left(\Blk\mathcal{S}\right)\mr{d}\Ohat=
\int_{\Omega_{\kappa}}\frac{\partial\mathcal{S}}{\partial t}\mr{d}\Omega,
\end{split}
\end{equation*}
and is therefore mimetic of the appropriate continuous term, this implies that 
\begin{equation*}
\frac{\mr{d}\vecJk}{\mr{d}\tau}+\sum\limits_{l=1}^{3}\Dxilsca\matBlksca\onesca = \bm{0},\qquad
\sum\limits_{l=1}^{3}\Dxilsca\matAlmksca\onesca = \bm{0},\qquad m = 1,2,3,
\end{equation*}
which are nothing more than discretizations of the continuous GCL~\eqref{eq:GCLtime} and~\eqref{eq:GCLspace}.

Assuming the above hold, 
\small
\begin{multline}\label{eq:skewEulersemi3}
\onesca\Tr\Msca\left[\frac{\mr{d}\matJksca\sk}{\mr{d}\tau}+\sum\limits_{l=1}^{3}\Dxilsca\matBlksca\sk\right]
=\\
\overbrace{\wk\Tr\M\bm{SAT}_{\tau}+\sum\limits_{l=1}^{3}\left[\onesca\Tr\Exilsca\matBlksca\wk\Tr\qk-\one\Tr\left(\Exil\matBlk\right)\circ\matQsc{\qk}{\qk}\wk\right]}^{IV}+\\
\overbrace{\wk\Tr\M\bm{SAT}_{\xil}+\sum\limits_{l,m=1}^{3}\left[\onesca\Tr\Exilsca\matAlmksca\psixmk-\one\Tr\left(\Exil\matAlmk\right)\circ\matFxmsc{\qk}{\qk}\wk\right]}^{V}.
\end{multline}
\normalsize
Next, the term $IV$ reduces to
\begin{equation}\label{eq:termIIII}
\begin{split}
&IV=\sum\limits_{l=1}^{3}\left[\onesca\Tr\Exilsca\matBlksca\wk\Tr\qk\right.\\
&+\wk\Tr\left(\matBlk\Rxilalpha\Tr\Pxilortho\Rxilbeta\right)\circ\matQsc{\qk}{\qtwolmone}\one\\
&\left.-\wk\Tr\left(\matBlk\Rxilbeta\Tr\Pxilortho\Rxilalpha\right)\circ\matQsc{\qk}{\qtwol}\one\right],
\end{split}
\end{equation}
where the fact that $\Exil$ is diagonal has been used throughout. Next, the term $V$ reduces to  
\begin{equation}\label{eq:termV}
\begin{split}
&V=\sum\limits_{l,m=1}^{3}\left[\onesca\Tr\Exilsca\matAlmksca\psixmk\right.\\
&+\left(\matAlmk\Rxilalpha\Tr\Pxilortho\Rxilbeta\right)\circ\matQsc{\qk}{\qtwolmone}\one\\
&\left.-\left(\matAlmk\Rxilbeta\Tr\Pxilortho\Rxilalpha\right)\circ\matFxmsc{\qk}{\qtwol}\one\right].
\end{split}
\end{equation}

Substituting~\eqref{eq:termIIII} and~\eqref{eq:termV} into~\eqref{eq:skewEulersemi3} and rearranging and summing over all elements gives
\begin{equation}\label{eq:skewEulersemi4}
\begin{split}
&\sum\limits_{\kappa=1}^{K}\onesca\Tr\Msca\frac{\mr{d}\matJksca\sk}{\mr{d}\tau}
=\sum\limits_{\kappa=1}^{K}\Bigg\{\\
&\sum\limits_{l=1}^{3}\left[\onesca\Tr\Exilsca\matBlksca\phik+\wk\Tr\left(\matBlk\Rxilalpha\Tr\Pxilortho\Rxilbeta\right)\circ\matQsc{\qk}{\qtwolmone}\one\right.\\
&\left.-\wk\Tr\left(\matBlk\Rxilbeta\Tr\Pxilortho\Rxilalpha\right)\circ\matQsc{\qk}{\qtwol}\one\right]\\
&+\sum\limits_{l,m=1}^{3}\left[\onesca\Tr\Exilsca\matAlmksca\psixmk\right.\\
&+\left(\matAlmk\Rxilalpha\Tr\Pxilortho\Rxilbeta\right)\circ\matQsc{\qk}{\qtwolmone}\one\\
&-\left.\left(\matAlmk\Rxilbeta\Tr\Pxilortho\Rxilalpha\right)\circ\matFxmsc{\qk}{\qtwol}\one\right]
\Bigg\},
\end{split}
\end{equation}
The coupling terms each have counter-parts from abutting elements, thus, for example
\begin{equation*}
\begin{split}
&\bm{CT} = \wk\Tr\left(\matBlk\Rxilalpha\Tr\Pxilortho\Rxilbeta\right)\circ\matQsc{\qk}{\qtwolmone}\one\\
&-\wtwolmone\Tr\left(\matBltwolmone\Rxilbeta\Tr\Pxilortho\Rxilalpha\right)\circ\matQsc{\qtwolmone}{\qk}\one\\
&=\phik\Tr\matBlksca\Rxilalphasca\Tr\Pxilorthosca\Rxilbetasca\onesca
-\onesca\Tr\matBlksca\Rxilalphasca\Tr\Pxilorthosca\Rxilbetasca\phitwolmone\\
&=\phik\Tr\matBlksca\Rxilalphasca\Tr\Pxilorthosca\Rxilbetasca\onesca
-\phitwolmone\Tr\matBltwolmonesca\Rxilbetasca\Tr\Pxilorthosca\Rxilalphasca\onesca,
\end{split}
\end{equation*}
where Theorem~\ref{thrm:telescope} has been used. Moreover,  advantage has been 
taken of 1) the fact that terms like $\Rxilalphasca\Tr\Pxilorthosca\Rxilbetasca$ are diagonal and therefore commute with 
the diagonal matrices containing the metric terms, 2) such terms only pick of metrics at the interface, and 3) the metric terms match at the interface.

The result of the above analysis is that the surface terms in~\eqref{eq:skewEulersemi4} cancel out at interior elements; thus, 
on a periodic domain 
\begin{equation}\label{eq:skewEulersemi5}
\begin{split}
&\sum\limits_{\kappa=1}^{K}\onesca\Tr\Msca\frac{\mr{d}\matJksca\sk}{\mr{d}\tau} = 0.
\end{split}
\end{equation}
For non-periodic domains, the remaining terms are surface terms at the boundary of the domain and appropriate entropy conservative/stable 
SATs need to be employed.

\section*{Acknowledgments}
%=================
N. Yamaleev and J. Lou gratefully acknowledge the support provided by Army Research Office through grant W911NF-17-1-0443.
We would also like to thank Old Dominion University for the use of the Turing High Performance Computing cluster.

\bigskip

\bibliographystyle{siamplain}

\bibliography{references}

\end{document}